\documentclass[10pt,twoside,reqno]{amsart}%
\numberwithin{equation}{section}%
\numberwithin{table}{section}%
\numberwithin{figure}{section}%
\usepackage[T1]{fontenc}%
\usepackage[utf8]{inputenc}%
\usepackage[a4paper,left=2.5cm,right=2.5cm,top=2.8cm,bottom=3.5cm]{geometry}%
\usepackage{hyperref,tikz,amsthm}%
\usepackage{lmodern}
\title{Aspects of Coulomb gases}%
\author{Djalil \textsc{Chafaï}}%
\address{CEREMADE, Université Paris-Dauphine, Université PSL, CNRS, 75016 Paris, France}%
\urladdr{http://djalil.chafai.net/}%
\date{Summer 2021, compiled \today.}
\newtheorem{theorem}{Theorem}[section]%
\newtheorem{corollary}[theorem]{Corollary}%
\newtheorem{lemma}[theorem]{Lemma}%
\newtheorem{remark}[theorem]{Remark}%
\newcommand{\weak}{\underset{n\to\infty}{\overset{\mathcal{C}_b}{\longrightarrow}}}
\begin{document}
\begin{abstract}
  Coulomb gases are special probability distributions, related to potential
  theory, that appear at many places in pure and applied mathematics and
  physics. In these short expository notes, we focus on some models, ideas,
  and structures. We present briefly selected mathematical aspects, mostly
  related to exact solvability and first and second order global asymptotics.
  A particular attention is devoted to two-dimensional exactly solvable models
  of random matrix theory such as the Ginibre model. Thematically, these notes
  lie between probability theory, mathematical analysis, and statistical
  physics, and aim to be very
  accessible. %No knowledge in potential theory is needed.
  They form a contribution to a volume of the \emph{Panoramas et Synthèses}
  series around the workshop \emph{États de la recherche en mécanique
    statistique}, organized by Société Mathématique de France, held at
  Institut Henri Poincaré, Paris, in the fall of 2018
  (\url{https://statmech2018.sciencesconf.org/}).
\end{abstract}
\maketitle
{\footnotesize\tableofcontents}
%
% \textbf{Disclaimer:}

% These notes form a partial, superficial, and subjective survey focusing on
% selected elementary and explicit aspects; many proofs are omitted or only
% sketched; the bibliography is incomplete and biased; last but not least
% typos and mistakes are randomly spread out along the text.

\medskip

There are several introductory texts around Coulomb gases. We refer for
instance to \cite{MR2168344,MR2514781,MR1677884,MR2760897,MR2920518,MR2641363}
for the relation to random matrices, to \cite{MR3309890} for the relation to
analysis and Ginzburg\,--\,Landau vortices, to \cite{MR3443870,phdavid} and
references therein for a relation to geometry, to \cite{butheze} and
references therein for a relation to random polynomials, to \cite{rougerie}
for a relation to Fock\,--\,Hartree quantum theory and Bose\,--\,Einstein
condensates, to \cite{MR3888701} and \cite{lewin-mega} for an overview from a
mathematical analysis/physics perspective, and to \cite{PhysRevA.99.021602}
for a statistical physics point of view.

% Note diverses suite à l'exposé de SS au MEGA
% Normalisation de beta N^{2/d-1} H avec H pas normalisé
% CLT pour beta-Ginibre :
% pour beta = 2 : Rider-Virag, Ameur-Hedenmalm-Makarov
% pour beta qlq : Leblé-Serfay et son equivalent Bourgade-BNY, échelle
% mécoscopique
% Serfaty : dim3 aussi, demande plus de regularité, mais descent à l'échelle
% micro plus rigide que le TCL pour variables iid, c'est à dire variance plus
% petite Pour les CLT, développement à l'ordre 2 de la log Z (énergie libre),
% dispo à tout ordre en dim 1 pour logas.
% Presque additivité avec erreur de l'ordre de la surface, permet de ramener
% le problème aux densités constantes
   
\section*{Notation}

The Euclidean norm of $x=(x_1,\ldots,x_d)\in\mathbb{R}^d$ is
$|x|=\sqrt{x_1^2+\cdots+x_d^2}$. It is the modulus if $d=2$ with the
identification $\mathbb{C}=\mathbb{R}^2$. We set
$\mathrm{i}=(0,1)\in\mathbb{C}$. The real and imaginary parts of
$z\in\mathbb{C}$ are denoted $\Re z$ and $\Im z$. The Lebesgue measure is
denoted $\mathrm{d}x$. Let $(E,\tau)$ be a topological space with Borel
$\sigma$-field $\mathcal{B}(E)$. We denote by $\mathcal{C}_b(E,\mathbb{R})$
the set of bounded continuous functions $E\to\mathbb{R}$, and by
$\mathcal{M}_1(E)$ the set of probability measures on $(E,\mathcal{B}(E))$. If
$\mu_1,\mu_2,\ldots,\mu$ are in $\mathcal{M}_1(\mathbb{R}^d)$, then
$\lim_{n\to\infty}\mu_n=\mu$ weakly, denoted
\[
  \mu_n\weak\mu,
  \text{  when for all }f\in\mathcal{C}_b(E,\mathbb{R})\text{ we have  }%
  \lim_{n\to\infty}\int f\mathrm{d}\mu_n=\int f\mathrm{d}\mu.
\]
This defines a sequential topology on $\mathcal{M}_1(E)$, giving a Borel
$\sigma$-field $\mathcal{B}(\mathcal{M}_1(E))$. A random probability measure
on $E$ is a random variable taking values in $\mathcal{M}_1(E)$. By $X\sim\mu$
we mean that the random variable $X$ has law $\mu$. We denote by
$\overset{\mathrm{d}}{=}$ and $\overset{\mathrm{d}}{\longrightarrow}$ the
equality and the convergence in law respectively.

\section{Coulomb electrostatics and equilibrium measures}

The Coulomb kernel is identical to the Newton kernel. Mathematically,
potential theory deals with the analysis of the Laplacian and its Green
function and the behavior of harmonic functions. In some sense, it emerges
naturally from the gravitation theory of Johannes Kepler and Isaac Newton, as
well as from the modeling of electrostatics, namely the study of the
distribution of static electric charges on conductors and their interactions.
From this last point of view, it takes its historical roots in the works of
Charles-Augustin de Coulomb, Joseph-Louis Lagrange, and Carl Friedrich Gauss.
Some mathematical parts of potential theory were developed later on by \,--\,
among others \,--\, Johann Peter Gustav Lejeune Dirichlet, Victor Gustave
Robin, Henri Poincaré, David Hilbert, Charles-Jean de La Vallée Poussin
\cite{zbMATH03031018}, Marcel Brelot \cite{MR0259146},
% Marcel Riesz and his student
Otto Frostman \cite{frostman}, Oliver Dimon Kellogg \cite{zbMATH03244821},
Gustave Choquet \cite{MR0080760}, and their followers. Deep links with Markov
processes and probability theory were explored in particular by Joseph Leo
Doob in \cite{MR1814344}, Gilbert Hunt, Claude Delacherie and Paul-André Meyer
\cite{dellacheriemeyer}, and their followers. Nowadays the basic objects and
structures of potential theory appear at many places in mathematics and
physics, providing in general a very useful electrostatic modeling or
interpretation.

\bigskip

\emph{Cet aspect probabiliste, que Brelot regrettait tant d’être arrivé trop
  tard pour maîtriser, est sans doute le plus bel exemple d’interaction
  féconde entre deux théories : La théorie du potentiel, née dans le ciel
  (Kepler 1618, Newton 1665) et la théorie des probabilités née d’un coup de
  dés (Pascal 1654), donc presque simultanément, devaient après trois siècles
  et des petits pas l’une vers l’autre (Wiener 1923, puis P. Levy, Doob),
  prendre avec G. Hunt (1957) pleinement conscience que leurs parties les plus
  vivaces ne sont que deux faces complémentaires d’un même bel objet, et qu’on
  ne peut bien comprendre l’une sans connaître l’autre (le traité
  Dellacherie-P.A. Meyer veut en donner la preuve).}
\begin{flushright}
  Gustave Choquet, \emph{La vie et l'oeuvre de {M}arcel {B}relot
    (1903\,--\,1987)} \cite{Choquet1990}.
\end{flushright}

\bigskip

\begin{table}
  \begin{tabular}{r|l}
    $\cdots$ & $\cdots$\\
    1660 & Newton, \ldots\\
    1770 & Coulomb, \ldots\\
    1800 & Gauss, \ldots\\
    1870 & Boltzmann, Gibbs, \ldots\\
    1900 & Thomson, de la Vallée Poussin, \ldots\\
    1920 & Fock, Hartree, \ldots\\
    1930 & Wigner, Wishart, \ldots\\
    1940 & Doob, Onsager, \ldots\\
    1950 & Choquet, Hunt, Wigner, \ldots\\
    1960 & Dyson, Ginibre, Mehta, Selberg, \ldots\\
    1970 & Kosterlitz, Landkof, Pastur, Thouless, \ldots\\
    1980 & Deift, Laughlin, Lebowitz, Saff, Voiculescu, \ldots\\
    1990 & Ben Arous, Edelman, Guionnet, Johansson, \ldots\\ %Lions, Kostlan\\
    2000 & Forrester, Erdős, Lewin, Serfaty, Yau, \ldots\\%, Bourgade, Bauerschmidt\\
    $\cdots$ & $\cdots$
  \end{tabular}
  \caption{The arrow of time and some of the main actors mentioned in the text
    or in the references. As mentioned in \cite{MR3434251}, the \emph{Stigler
      law of eponymy} states that ``\emph{No scientific discovery is named
      after its original discoverer.}'', attributed by Stephen Stigler to
    Robert K.~Merton. This is also known as the \emph{Arnold principle} by
    some people.}
  % Other name: The Arnold Principle: If a notion bears a personal name, then
  % this name is
  % not the name of the
  % discoverer. The Berry Principle: The Arnold Principle is applicable to
  % itself.
  % Source: “On teaching mathematics” by V.I. Arnold –
  % http://pauli.uni-muenster.de/~munsteg/arnold.html – an extended text of
  % the address at the discussion on teaching of mathematics in Palais de
  % Découverte in Paris on 7 March 1997.
\end{table}

  \bigskip
  
Let $d\geq1$. The \emph{Coulomb kernel} $g$ in $\mathbb{R}^d$ is given by
$g(0)=+\infty$, and, for all $x\in\mathbb{R}^d\setminus\{0\}$,
\begin{equation}\label{eq:coulomb}
  g(x)=
  \begin{cases}
    % -|x| & \text{if $d=1$}\\
    \displaystyle\log\frac{1}{|x|} & \text{if $d=2$,}\\[1.5em]
    \displaystyle\frac{1}{(d-2)|x|^{d-2}} & \text{if not}.
  \end{cases}
\end{equation}
We say that $(x,y)\mapsto G(x,y)=g(x-y)$ is the Green function of the Laplace
operator $\Delta=\partial^2_1+\cdots+\partial^2_d$, and $g$ is the fundamental
solution of the Poisson equation, see for instance \cite[Theorem
6.20]{MR1817225}. Indeed, denoting $\delta_0$ the Dirac mass at the origin, we
have, in the sense of Schwartz distributions,
\begin{equation}\label{eq:laplace}
  -\Delta g=c_d\delta_0
  \quad\text{and}\quad
  c_d=d\omega_d=\frac{2\pi^{d/2}}{\Gamma(d/2)}
\end{equation}
where $\omega_d$ is the volume of the unit ball (its surface is $d\omega_d$).
The case $d=3$ is physical for electrostatics modeling in the ambient space.
The case $d=2$ also appears at many places in mathematical physics. The case
$d=1$ serves historically as a toy model, less singular but exactly solvable.

For simplicity, we suppose from now on that $d\geq2$.

\begin{figure}[htbp]
  \centering
  \begin{tikzpicture}
    \draw[very thin, dotted] (0,1) -- (1,1);
    \draw[very thin, dotted] (1,0) -- (1,1);
    \draw[very thin, dotted] (0,-1) -- (1,-1);
    \draw[very thin, dotted] (1,0) -- (1,-1);
    \draw[->] (0,0) -- (3.5,0) node[right] {$|x|$};
    \draw[->] (0,-3) -- (0,4.5) node[above] {$g(x)$};
    \draw[domain=0:3,smooth,variable=\x,black,solid,very thick] plot ({\x},{-\x});     
    \draw[domain=0.1:3.25,smooth,variable=\x,black,dashed,very thick] plot
    ({\x},{-ln(\x)});
    \draw[domain=0.5:3.25,smooth,variable=\x,black,dotted,very thick] plot
    ({\x},{1/\x^2});
    \draw (1,0) node [below] {$1$};
    \draw (0,0) node [left] {$0$};
    \draw (0,1) node [left] {$1$};
    \draw (0,-1) node [left] {$-1$};
    \draw[black, dotted, very thick] (2.5,2.5) -- (3,2.5);
    \draw[black, dashed, very thick] (2.5,3) -- (3,3);
    \draw[black, solid, very thick] (2.5,3.5) -- (3,3.5);
    \draw (3,2.5) node [right] {$d=3$};
    \draw (3,3) node [right] {$d=2$};
    \draw (3,3.5) node [right] {$d=1$};    
  \end{tikzpicture}
  \caption{Coulomb kernel in dimension $1$ (solid line) $2$ (dotted line) and
    $3$ (dashed line).\label{fi:g}}
\end{figure}
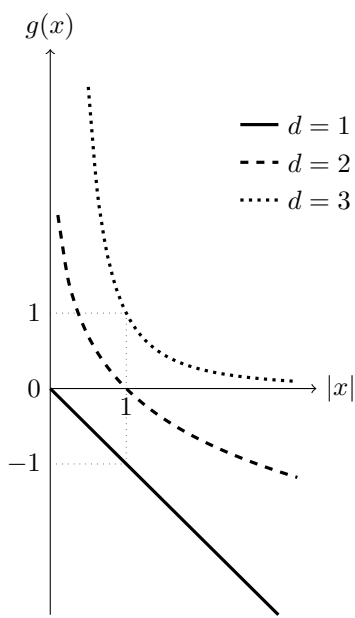

For all $\mu\in\mathcal{M}_1(\mathbb{R}^d)$ such that
$\log(1+\left|\cdot\right|)\mathbf{1}_{d=2}\in\mathrm{L}^1(\mu)$, the
\emph{Coulomb energy} of $\mu$ is
\begin{equation}\label{eq:E}
  \mathcal{E}(\mu)
  =\frac{1}{2}\iint g(x-y)\mathrm{d}\mu(x)\mathrm{d}\mu(y)
  \in(-\infty,+\infty].
\end{equation}
Note that if $d=2$ then $\mathcal{E}(\mu)=+\infty$ if $\mu$ has a Dirac mass.
If $\mu$ models the distribution of unit charges (say electrons) in
$\mathbb{R}^d$ then $\mathcal{E}(\mu)$ is the electrostatic self-interaction
energy of the configuration $\mu$.

We say that a Borel set $B\in\mathcal{B}(\mathbb{R}^d)$ is of \emph{positive
  capacity} when $\mathrm{supp}(\mu)\subset B$ and $\mathcal{E}(\mu)<\infty$
for some $\mu\in\mathcal{M}_1(\mathbb{R}^d)$, and is of \emph{zero capacity}
when it does not carry a probability measure $\mu$ with
$\mathcal{E}(\mu)<\infty$.

For all $\mu\in\mathcal{M}_1(\mathbb{R}^d)$ with
$\log(1+\left|\cdot\right|)\mathbf{1}_{d=2}\in\mathrm{L}^1(\mu)$, the
\emph{Coulomb potential} of $\mu$ at $x\in\mathbb{R}^d$ is defined by
\[
  U_\mu(x)
  =\int g(x-y)\mathrm{d}\mu(y)
  =(g*\mu)(x).
\]
We have $U_\mu(x)\in(-\infty,+\infty]$, and $U_\mu(x)=+\infty$ if $\mu$ has a
Dirac mass at point $x$. We also have
\begin{equation}\label{eq:EU}
  \mathcal{E}(\mu)=\frac{1}{2}\int U_\mu(x)\mathrm{d}\mu(x).
\end{equation}
Since $g\in\mathrm{L}^1_{\mathrm{loc}}(\mathbb{R}^2,\mathrm{d}x)$, the
Fubini\,--\,Tonelli theorem gives
$U_\mu\in\mathrm{L}^1_{\mathrm{loc}}(\mathbb{R}^2,\mathrm{d}x)$, hence
$U_\mu<+\infty$ almost everywhere. Moreover $U_\mu=g*\mu$, and, in the sense
of Schwartz distributions, we get from \eqref{eq:laplace} that
\begin{equation}\label{eq:Umu}
  \Delta U_\mu=-c_d\mu.
\end{equation}
In particular, this gives the formula
\begin{equation}\label{eq:edelta}
  \mathcal{E}(\mu)
  =\frac{1}{2}\int U_\mu\mathrm{d}\mu
  =-\frac{1}{2c_d}\int U_\mu\Delta U_\mu\mathrm{d}x.
\end{equation}

When $d\geq3$, the functional $\mathcal{E}$ does not take negative values on
probability measures because $g\geq0$. However, when $d=2$, the functional
$\mathcal{E}$ may take negative values on compactly supported probability
measures, due to the change of sign of $g$ when $d=2$ inside and outside the
unit ball. For instance if $\mu_r$ is the uniform law on the circle
$\{x\in\mathbb{C}:|x|=r\}$ of radius $r>0$ then, for all $x\in\mathbb{C}^2$,
\begin{equation}\label{eq:Uunif}
  U_{\mu_r}(x)=-\log(r)\mathbf{1}_{|x|\leq r}-\log|x|\mathbf{1}_{|x|>r},
  \quad\text{and}\quad
  \mathcal{E}(\mu_r)=-\frac{\log(r)}{2}, 
\end{equation}
which is negative if $r>1$. See for instance \cite[(0.5.5) and
(I.1.6)]{MR1485778} for these computations. Similary, if $\mu_R$ is the
uniform law on the disc $\{x\in\mathbb{C}:|x|\leq R\}$ of radius $R>0$ then we
find that for all $x\in\mathbb{C}^2$,
\begin{equation}\label{eq:UunifD}
  U_{\mu_R}(x)=-\frac{1}{2}\Big(\frac{|x|^2}{R^2}-1+2\log R\Bigr)\mathbf{1}_{|x|\leq R}-\log|x|\mathbf{1}_{|x|>R},
  \quad\text{and}\quad
  \mathcal{E}_{\mu_R}=\frac{1}{4}-\log(R).
\end{equation}

The functionals $U$ and $\mathcal{E}$ extend to signed measures. If
$\eta=\mu-\nu$ where $\mu$ and $\nu$ are two compactly supported probability
measures on $\mathbb{R}^d$, then $U_\eta$ vanishes at infinity and an
integration by parts gives
\begin{equation}\label{eq:c2c}
  \mathcal{E}(\eta)
  =\frac{1}{2}\int U_\eta\mathrm{d}\eta
  =-\frac{1}{2c_d}\int U_\eta\Delta U_\eta\mathrm{d}x
  =\frac{1}{2c_d}\int|\nabla U_\eta|^2\mathrm{d}x,
\end{equation}
see \cite{MR3309890}. This shows that $\mathcal{E}$ does not take negative
values on signed measures with total mass zero.

The right hand side of \eqref{eq:c2c} is the ``carré du champ'' in potential
theory \cite{MR0448158,MR521767} while $-\nabla U_\mu$ is the electric field
\,--\, ``champ électrique'' in French \,--\, generated by the configuration of
charges $\mu$.
% FIXME: for compactly supported measures. There is a boundary term in the
% IBP when d=2 since $U_\mu\nabla U_\mu$ does not go to zero at infinity,
% While in dimension 3 it works. We also have a problem in dimension 1.
% FIXME: impose from the beginning that dim >1

Let us introduce now $V:\mathbb{R}^d\to(-\infty,+\infty]$ such that
(we say then that $V$ is an \emph{admissible potential}):
\begin{itemize}
\item the function $V$ is lower semi-continuous;
\item the set $\{x\in\mathbb{R}^d:V(x)<+\infty\}$ has positive capacity;
\item the function $V$ is not beaten by the Coulomb kernel at infinity, namely
  \begin{equation}\label{eq:Vweak}
    \varliminf_{|x|\to\infty}\left(V(x)-\log\left|x\right|\mathbf{1}_{d=2}\right)>-\infty.
  \end{equation}
\end{itemize}
The \emph{electrostatic energy} with \emph{external potential} $V$ is defined
from $\mathcal{M}_1(\mathbb{R}^d)$ to $(-\infty,+\infty]$ by
\begin{equation}\label{eq:EV}
  \mathcal{E}_V(\mu)=
  \frac{1}{2}\iint\left(g(x-y)+V(x)+V(y)\right)\mu(\mathrm{d}x)\mu(\mathrm{d}y).
\end{equation}
This makes sense since the function under the double integral is bounded below
on $\mathbb{R}^d\times\mathbb{R}^d$ thanks to \eqref{eq:Vweak}. Finally, for
all $\mu\in\mathcal{M}_1(\mathbb{R}^d)$, if both
$\log(1+\left|\cdot\right|)\mathbf{1}_{d=2}$ and $V$ are in
$\mathrm{L}^1(\mu)$, then
\begin{equation}\label{eq:EVbis}
  \mathcal{E}_V(\mu)
  =\mathcal{E}(\mu)+\int V(x)\mathrm{d}\mu(x).%\in\mathbb{R}.
\end{equation}
The external potential plays typically the role of a confinement.

The convexity of the quadratic form $\mathcal{E}_V$ is related to a Bochner
positivity of the kernel $g$, see
\cite{MR0350027,MR1746976,MR3262506,zbMATH07103761}. Indeed for all
$\lambda\in(0,1)$ and $\mu,\nu\in\mathcal{M}_1(\mathbb{R}^d)$ with
$\mathcal{E}(\mu)<+\infty$ and $\mathcal{E}(\nu)<+\infty$ and
$V\in\mathrm{L}^1(\mu)\cap\mathrm{L}^1(\nu)$,
\[
  \frac{\lambda\mathcal{E}_V(\mu)+(1-\lambda)\mathcal{E}_V(\nu)
    -\mathcal{E}_V(\lambda\mu+(1-\lambda)\nu)}
  {\lambda(1-\lambda)}
  =\mathcal{E}(\mu-\nu)
  =\frac{1}{2c_d}\int|\nabla U_{\mu-\nu}|^2\mathrm{d}x\geq0.
\]
%FIXME: note that since $\mu-\nu$ is a difference of two compactly supported
%probability masures we have that the IBP has no boundary term.

%\begin{figure}[htbp]
%  \centering
%  \begin{tikzpicture}
%    \draw[->] (0,0) -- (3.5,0) node[right] {$|x|$};
%    \draw[->] (0,-1.5) -- (0,1.5) node[above] {$U_{\mu_r}(x)$};
%    \draw[domain=1:3.25,smooth,variable=\x,black,solid,very thick] plot
%    ({\x},{-ln(\x)});
%    \draw[domain=0:1,smooth,variable=\x,black,solid,very thick] plot
%    ({\x},{0});
%    \draw (0,0) node [above left] {$-\log(r)$};
%    \draw (1,0) node [below left] {$r$};
%    \draw (0,0) node [below right] {$0$};
%  \end{tikzpicture}
%  \caption{The Coulomb potential of a uniform distribution on the disc of
%    radius $r$.\label{fi:Umur}}
%\end{figure}

We are now ready for the general concept of equilibrium measure and its
properties. The following couple of theorems is a classic in potential theory.
For a proof, we refer for instance to the books
\cite{MR0350027,MR3308615,MR1485778,MR3309890} and to the articles
\cite{MR1465640,MR3262506,MR3888701}.

\begin{theorem}[Equilibrium measure]\label{th:mustar}
  The following properties hold true:
  \begin{enumerate}
  \item $\mathcal{E}_V$ is strictly convex on its domain, is lower
    semi-continuous, with compact level sets;
  \item $\inf_{\mathcal{M}_1(\mathbb{R}^d)}\mathcal{E}_V<+\infty$;
  \item there exists a unique $\mu_V\in\mathcal{M}_1(\mathbb{R}^d)$, called the
    \emph{equilibrium measure}, such that
    \[
      \mathcal{E}_V(\mu_V)
      =\inf_{\mu\in\mathcal{M}_1(\mathbb{R}^d)}\mathcal{E}_V(\mu)
      \quad\text{in other words}\quad
      \mu_V=\arg\min_{\mathcal{M}_1(\mathbb{R}^d)}\mathcal{E}_V.
    \]
  \end{enumerate}
\end{theorem}

Some examples of equilibrium measures are gathered in Table \ref{tb:mustar}.

\begin{table}
  \centering
  \setlength{\arrayrulewidth}{.8pt}  
  \renewcommand{\arraystretch}{1.8}  
  \begin{tabular}{r|l|l}
    Dimension $d$ 
    & Potential $V$ %
    & Equilibrium measure $\mu_V$\\\hline%
    $\geq1$%
    & $\infty\mathbf{1}_{\left|\cdot\right|>r}$%
    & Uniform on sphere $\{x\in\mathbb{R}^d:|x|=r\}$\\    
    $\geq1$%
    & finite and $\mathcal{C}^2$%
    & With density $\frac{\Delta V}{c_d}$ on the interior of its support\\
    &&(absolutely continuous part of $\mu_*$)\\
    $\geq1$ & $\frac{1}{2}\left|\cdot\right|^2$%
    & Uniform on unit ball with density $\frac{\mathbf{1}_{\left|\cdot\right|\leq1}}{\omega_d}$\\
    (Ginibre) $2$ & $\frac{1}{2}\left|\cdot\right|^2$%
    & Uniform on unit disc with density
      $\frac{\mathbf{1}_{\left|\cdot\right|\leq1}}{\pi}$\\
    (Spherical) $2$%
    & $\frac{1}{2}\log(1+\left|\cdot\right|^2)$%
    & Heavy-tailed with density $\frac{1}{\pi(1+\left|\cdot\right|^2)^2}$\\
    (CUE) $2$%
    & $\infty\mathbf{1}_{([a,b]\times\{0\})^c}$%
    & Arcsine on $[a,b]\times\{0\}$, density
      $s\mapsto\frac{\mathbf{1}_{s\in[a,b]}}{\pi\sqrt{(s-a)(b-s)}}$\\
    (GUE) $2$%
    & $\frac{\left|\cdot\right|^2}{2}\mathbf{1}_{\mathbb{R}\times\{0\}}+\infty\mathbf{1}_{(\mathbb{R}\times\{0\})^c}$
    & Semicircle on $[-2,2]\times\{0\}$, density $s\mapsto\frac{\sqrt{4-s^2}}{2\pi}\mathbf{1}_{s\in[-2,2]}$
  \end{tabular}
  \medskip
  \caption{\label{tb:mustar}Basic examples of equilibrium measures. Some other
    examples are given in Section \ref{se:more}. The last four examples appear
    as limiting spectral distributions of random matrices. The last two
    examples are singular in the sense that the potential is infinite outside
    the real line. They appear as one-dimensional log-gases from random
    matrices. In this case, the equilibrium measure cannot be deduced as a
    specialization of the second example, and its computation is a bit more
    subtle, see for instance \cite{MR1485778}.}
\end{table}

\begin{theorem}[Properties of the equilibrium measure]\label{th:mustarprop}
  The following properties hold true:
  \begin{enumerate}
  \item the equilibrium measure $\mu$ is compactly supported if 
      \begin{equation}\label{eq:Vstrong}
        \lim_{|x|\to\infty}\left(V(x)-\log\left|x\right|\mathbf{1}_{d=2}\right)=+\infty;
      \end{equation}
    \item the equilibrium measure $\mu_V$ has finite Coulomb energy
      $\mathcal{E}(\mu_V)\in\mathbb{R}$;
  \item we have $\mathrm{supp}(\mu_V)\subset\{x\in\mathbb{R}^d:V(x)\leq R\}$
    for some constant $R<\infty$;
  \item the following Euler\,--\,Lagrange equations hold:
    \begin{itemize}
    \item $U_{\mu_V}(x)+V(x)\leq c_V$ for all $x\in\mathrm{supp}(\mu_V)$,
    \item $U_{\mu_V}(x)+V(x)\geq c_V$ for all $x\in\mathbb{R}^d$ except on a set
      of zero capacity,
    \end{itemize}
    where $c_V$ is a quantity called the \emph{modified Robin constant}
    defined by
    \[
      c_V=\mathcal{E}(\mu_V)-\int
      V\mathrm{d}\mu_V.
    \]
    In particular, for all $x\in\mathrm{supp}(\mu_V)$ except on a set of zero
    capacity, we have
    \[
      U_{\mu_V}(x)+V(x)=c.
    \]
    % \item
    %   if $V$ is finite with Lipschitz continuous derivative then the
    %   absolutely continuous part of $\mu_V$ with respect to the Lebesgue measure
    %   has density given on the interior of its support by
    %   \[
    %     \frac{\Delta V}{c_d}.
    %   \]
    In particular, we have the equality in the sense of distributions
    \[
      \mu_V=\frac{\Delta V}{c_d},
    \]
    and the interior of $\mathrm{supp}(\mu_V)$ does not intersect
    $\{\Delta V<0\}$.
  \end{enumerate}    
\end{theorem}

\begin{remark}[Logarithmic kernels and Riesz kernels]\label{rm:logriesz}\ %
  \begin{itemize}
  \item The logarithmic kernel in dimension $d$ is given by
    \[
    -\log\left|x\right|,\quad x\in\mathbb{R}^d,x\neq0;
    \]
  \item The Riesz kernel $k_s$ in $\mathbb{R}^d$ with parameter $s>0$ is given
    by
    \[
      \frac{1}{s|x|^s},\quad x\in\mathbb{R}^d,x\neq0.
    \]
  \item The Coulomb kernel in dimension $d\neq2$ matches the Riesz kernel with
    $s=d-2$;
  \item The logarithmic kernel for all $d\geq1$ can be seen as the Riesz
    kernel with $s=0$. Indeed, it suffices to remove the singularity in the
    sense that for all $x\in\mathbb{R}^d$ with $x\neq0$,
    \[
	    \lim_{s\to0}\frac{|x|^{-s}-1}{s}=\partial_{s=0}|x|^{-s}=-\log|x|.
    \]
    In particular the Coulomb kernel in dimension $d=2$ is the Riesz kernel
    with $s\to0$.
  \item For all $\alpha\in(0,d)$, the Riesz kernel with $s=d-\alpha$ is the
    fundamental solution of the fractional Laplace operator 
    $\Delta_\alpha=\Delta^{\frac{\alpha}{2}}$, a
    Fourier multiplier, non-local operator if $\alpha\neq2$, see 
    \cite{MR3262506,MR3455593}.
  \end{itemize}
  We refer for instance to \cite{zbMATH07103761} for more analytic properties
  of these kernels and various applications.
\end{remark}

\section{Coulomb gases}

Let $d\geq2$, $n\geq1$, $\beta>0$, and let $g$ and $V$ as before. Suppose
moreover that $V$ is such that
\begin{equation}\label{eq:Vint}
  \int_{\mathbb{R}^d}\mathrm{e}^{-n\beta(V(x)-\log(1+|x|)\mathbf{1}_{d=2})}
  \mathrm{d}x<\infty.
\end{equation}
By using the fact that $g\geq0$ when $d\geq3$ and $|x-y|\leq(1+|x|)(1+|y|)$ when $d=2$, we
get then
\[
  Z_n=\int_{(\mathbb{R}^d)^n}\mathrm{e}^{-\beta
    E_n(x_1,\ldots,x_n)}
  \mathrm{d}x_1\cdots\mathrm{d}x_n<\infty
\]
where
\[
  E_n(x_1,\ldots,x_n)=n\sum_{i=1}^nV(x_i)+\frac{1}{2}\sum_{i\neq j}g(x_i-x_j).
\]
The \emph{Coulomb gas} $P_n$ is the Boltzmann\,--\,Gibbs probability measure on
$(\mathbb{R}^d)^n$ given by
\begin{equation}\label{eq:P}
  \mathrm{d}P_n(x_1,\ldots,x_n)
  =\frac{\mathrm{e}^{-\beta E_n(x_1,\ldots,x_n)}}{Z_n}
  \mathrm{d}x_1\cdots\mathrm{d}x_n.
\end{equation}
It models a ``gas of electrons'' in $\mathbb{R}^d$ of charge $1/n$, at
positions $x_1,\ldots,x_n$, inverse temperature $\beta n^2$, energy
$(1/n^2)E_n(x_1,\ldots,x_n)$, subject to Coulomb pair interaction and external
field of potential $V$, namely
\begin{equation}\label{eq:H}
  \beta E_n(x_1,\ldots,x_n)
  =\beta n^2\Bigr(\frac{1}{n}\sum_{i=1}^nV(x_i)+\frac{1}{n^2}\sum_{i<j}g(x_i-x_j)\Bigr).
\end{equation}
Beware that we should interpret $P_n$ as a way to model a random static
configuration of charged particles. We deal here with electrostatics rather
than with electrodynamics. The charged particles do not move and there is no
magnetic field. We have only an electric field.

In view of Remark \ref{rm:logriesz}, we could also define log-gases and Riesz
gases. We do not follow this idea in these notes, for simplicity and because
the Coulomb case is by far the most important in applications.

\subsection{One-dimensional log-gases as Coulomb gases}

The formula \eqref{eq:P} makes sense provided that $Z_n>0$. Actually the
integration in \eqref{eq:Vint} should be interpreted as with respect to the
trace of the Lebesgue measure or Hausdorff measure on
$\{V<+\infty\}\subset\mathbb{R}^d$. Similarly the integration in \eqref{eq:P}
should be interpreted as with respect to the trace of the Lebesgue measure or
Hausdorff measure on $\{V<+\infty\}^n\subset(\mathbb{R}^d)^n$. This allows to
incorporate in the Coulomb gas model \eqref{eq:P} the one-dimensional
log-gases of random matrix theory, by taking $d=2$ and $V=+\infty$ on $S^c$
where $S$ is a one-dimensional subset of $\mathbb{R}^2$, typically
$S=\{x\in\mathbb{R}^2:x_2=0\}$ or $S=\{x\in\mathbb{R}^2:|x|=1\}$. This
includes all beta Hermite/Laguerre/Jacobi ensembles, Gaussian
Unitary/Orthogonal/Simplectic Ensembles, etc. For instance, the famous
Gaussian Unitary Ensemble (GUE) corresponds to take $d=2$ and
\[
  x\in\mathbb{R}^2
  \mapsto
  V(x)=
  \begin{cases}
    \displaystyle\frac{|x|^2}{2}&\text{if $x\in S=\mathbb{R}\times\{0\}$}\\
    +\infty&\text{ if not}
  \end{cases}.
\]
For simplicity, we do not study further the one-dimensional log-gases, in
particular the ones coming from random matrix theory. We refer to the books
\cite{MR1677884,MR2129906,MR2168344,MR2514781,MR2641363,MR2760897,MR2808038}.
Actually, most of the models that we consider in the sequel are fully
dimensional in the sense that $V$ is finite everywhere. The simplest models
that we focus on are two-dimensional: beta-Ginibre gases.

\subsection{Beta-Ginibre gas}

The case $d=2$ is known as the \emph{two-dimensional one-component plasma}. We
call it the \emph{beta gas} for short. Its density with respect to the
Lebesgue measure on $(\mathbb{R}^2)^n=\mathbb{C}^n=\mathbb{R}^{2n}$ is
\begin{equation}\label{eq:beta}
  (z_1,\ldots,z_n)\in\mathbb{C}^n\mapsto
  \frac{\mathrm{e}^{-n\beta\sum_{j=1}^nV(z_j)}}{Z_n}\prod_{i<j}|z_i-z_j|^\beta.
\end{equation}
The quadratic potential case $V=\frac{1}{2}\left|\cdot\right|^2$ is sometimes
referred to as the \emph{beta-Ginibre gas}. In the special case $\beta=2$ and
$V=\frac{1}{2}\left|\cdot\right|^2$, that we call the \emph{Ginibre gas}, the
density of $P_n$ can be written as
\begin{equation}\label{eq:ginibre}
  (z_1,\ldots,z_n)\in\mathbb{C}^n\mapsto
  n^n\varphi_n(\sqrt{n}z_1,\ldots,\sqrt{n}z_n)
  \quad\text{with}\quad
  \varphi_n(z_1,\ldots,z_n)=\frac{\mathrm{e}^{-\sum_{j=1}^n|z_j|^2}}{\pi^n\prod_{k=1}^nk!}
  \prod_{i<j}|z_i-z_j|^2.
\end{equation}

The beta gas \eqref{eq:beta} with $V=\frac{1}{2}\left|\cdot\right|^2$ and
$\beta\in\{2,4,6,\ldots\}$ matches the squared modulus of the Laughlin wave
function of the fractional quantum hall effect \cite{Laughlin1987,MR2180172}.
The Ginibre gas \eqref{eq:ginibre} matches the density of the eigenvalues of
Gaussian random matrices \cite{MR0173726} (see Section \ref{se:ginibre} for
more details), the distribution of vortices in the Ginzburg\,--\,Landau
modeling of superconductivity \cite{MR3309890}, and rotating trapped fermions
in two dimensions \cite{PhysRevA.99.021602}. The beta gas \eqref{eq:beta} with
$\beta=2$ such as the Ginibre gas \eqref{eq:ginibre} has a determinantal
structure which provides exact solvability (see Section \ref{se:ginibre} for
more details).

% random polynomials. the distribution of the roots of Weyl random polynomials
% \cite{MR1690355,PhysRevLett.96.040405,article,MR3262481}

\subsection{From multivariate statistics to atomic physics}

Historically, Coulomb gases emerged in mathematical statistics in the years
1920/30 in the study of the spectral decomposition of empirical covariance
matrices of Gaussian samples. We speak nowadays about Laguerre ensembles and
Wishart random matrices. In the 1950s, Eugene P.\ Wigner discovered by
accident this model when reading a statistics textbook, and this led him to
use random matrices for the modeling of energy levels of heavy nuclei in
atomic physics, see for instance \cite{blog,MR2932622}. His work generated an
enormous trend of activity in statistical physics in the 1960s, with the works
of Gaudin, Mehta, Dyson, Ginibre, Marchenko, Pastur, among others. The term
\emph{Coulomb gas} is already in the abstract of the first seminal article of
Dyson \cite{MR0143556} and of Ginibre \cite{MR0173726}. The terms \emph{Fermi
  gas} and \emph{one-component plasma} are also used.

\subsection{The Wigner jellium and electrons in metals}

It turns out that Coulomb gases are related to another famous model of
mathematical physics also due to Wigner. Let $S\subset\mathbb{R}^d$ be compact
and let $\mu$ be a positive measure on $\mathbb{R}^d$ with
$\mu(\mathbb{R}^d)=\alpha>0$. The Coulomb gas with potential
\[
  V=
  \begin{cases}
    -\frac{1}{n}U_\mu&\text{on $S$},\\
    +\infty&\text{outside $S$}
  \end{cases}
\]
is known as a \emph{Wigner jellium} with background $\mu$, and is said to be
\emph{charge neutral} when $n=\alpha$. The background $\mu$ models a positive
charge $\alpha$ smeared out on $\mathrm{supp}(\mu)$. This model, or more
precisely its thermodynamic limit as $|S|\to+\infty$, was derived by Wigner in
1938 as an approximation of the Hartree\,--\,Fock quantum model in order to
model electrons in metals \cite{TF9383400678}, see also the 1904 pre-quantum
work by Thomson \cite{doi:10.1080/14786440409463107} on electrons and the
structure of atoms. Conversely, a Coulomb gas with smooth potential $V$ can be
seen as a jellium with background $\mu$ of density $\frac{\Delta V}{c_d}$.
From this point of view, by looking at \eqref{eq:UunifD}, the complex Ginibre
ensemble can be seen as a Jellium with full space Lebesgue background. The
measure $\mu$ is positive when $V$ is sub-harmonic (meaning that
$\Delta V\geq0$). If $V$ is not sub-harmonic then $\mu$ is no-longer positive
but we may interpret it as an opposite charge on the subset $\{\Delta V<0\}$.
We refer for instance to
\cite{chafai-garcia-zelada-jung-1d,doi:10.1063/1.5126724} for a bibliography
and a discussion. The term \emph{jellium} was apparently coined by Conyers
Herring, the smeared charge being viewed as a positive ``jelly'', see
\cite{hughes2006theoretical}.

\subsection{Random polynomials}

The Coulomb or log gases emerging from random matrix theory describe the law
of the eigenvalues of a random matrix, the roots of the characteristic
polynomial. This random polynomial has random dependent coefficients. We could
study the distribution of the roots of random polynomials with random
independent coefficients. Actually this question emerged from various fields
of research including algebraic and geometric analysis and number theory, for
instance with the works of Littlewood and Offord in the 1920s, independently
of the works of the statisticians on the spectral analysis of empirical
covariance matrices. The simplest model that we could imagine is a random
polynomial with independent and identically distributed coefficients. This
model is known as Kac polynomials, and the distribution of the roots was
computed in the Gaussian case by John Hammersley in 1956. There are several
other natural models of random polynomials and plenty of works on such models.
The gases emerging from these models are two-dimensional but differ from
Coulomb gases due to the presence in the energy of an additional non-quadratic
term with respect to the empirical measure. For more details, we refer for
instance to \cite{MR1418808,MR1160289,MR3127891,butheze} and references
therein. When the degree tends to infinity, such models give rise to random
analytic functions, see for instance \cite{MR3262481,MR2552864,butheze} and
references therein.

%FIXME: say a bit more and give a formula for the non-quadratic Coulomb gas

\section{First order global asymptotics and large deviations}

Let $P_n$ be the Coulomb gas as in \eqref{eq:P}. If $x_1,\ldots,x_n$ are
pairwise distinct elements of $\mathbb{R}^d$, which holds almost everywhere
with respect to $P_n$ in $(\mathbb{R}^d)^n$, we get from \eqref{eq:H} that
\[
  E_n(x_1,\ldots,x_n)=n^2\mathcal{E}^{\neq}_V(\mu_{x_1,\ldots,x_n})
\]
where
\[
  \mathcal{E}^{\neq}_V(\mu)=
  \int V\mathrm{d}\mu+\frac{1}{2}\iint\mathbf{1}_{u\neq v}g(u-v)
  \mathrm{d}\mu(u)\mathrm{d}\mu(v)
  \quad\text{and}\quad
  \mu_{x_1,\ldots,x_n}=\frac{1}{n}\sum_{i=1}^n\delta_{x_i}.
\]
The probability measure $P_n$ is \emph{exchangeable} in the sense that it is
invariant by permutation of the $n$ particles. The system is \emph{mean-field}
in the sense that each particle interacts with all the other particles via
their empirical measure. The density of $P_n$ at $(x_1,\ldots,x_n)$ is a
function of $\mu_{x_1,\ldots,x_n}$ and rewrites
\begin{equation}\label{eq:PnEneq}
  \frac{\exp\Bigr(-\beta n^2\mathcal{E}^{\neq}_V(\mu_{x_1,\ldots,x_n})\Bigr)}{Z_n}.
\end{equation}
In terms of asymptotic analysis, we expect that
$\mathcal{E}^{\neq}_V\approx\mathcal{E}_V$ as $n\to\infty$, and the Laplace
method suggests that under $P_n$, the empirical measure $\mu_{x_1,\ldots,x_n}$
concentrates as $n\to\infty$ around the minimizers of $\mathcal{E}_V$. Since
there is a unique minimizer known as the equilibrium measure $\mu_V$, we
expect that the empirical measure $\mu_{x_1,\ldots,x_n}$ under $P_n$ converges
towards $\mu_V$ as $n\to\infty$. More precisely, for all $n$, let us define
\[
  X_n=(X_{n,1},\ldots,X_{n,n})\sim P_n
  \quad\text{and}\quad
  \mu_n=\frac{1}{n}\sum_{k=1}^n\delta_{X_{n,k}}.
\]

\subsection{The large deviations principle}

For all Borel subsets $A\subset\mathcal{M}_1(\mathbb{R}^d)$,
$\mathbb{P}(\mu_n\in A)=P_n(\mu_{x_1,\ldots,x_n}\in A)$. The following theorem
translates mathematically the intuition above based on the Laplace principle:
$\mathbb{P}(\mu_n\in A)\approx_{n\to\infty}\mathrm{e}^{-\beta
  n^2\inf_A(\mathcal{E}_V-\mathcal{E}_V(\mu_V))}$. The difficulty lies in the
singularity of the Coulomb interaction.

\begin{theorem}[Large deviations principle]\label{th:LDP}
  We have
  \[
    \lim_{n\to\infty}\frac{\log Z_n}{\beta n^2}=-\mathcal{E}_V(\mu_V).
  \]
  Moreover the sequence ${(\mu_n)}_n$ satisfies to a large deviations principle
  of speed $n^2$ and good rate function $\mathcal{E}_V-\mathcal{E}_V(\mu_V)$,
  in other words for all Borel subset of
  $A\subset\mathcal{M}_1(\mathbb{R}^d)$, we have
  \[
    \mathcal{E}_V(\mu_V)-\inf_{\mathrm{int}(A)}\mathcal{E}_V
    \leq\varliminf_{n\to\infty}\frac{\log\mathbb{P}(\mu_n\in A)}{\beta n^2}
    \leq\varlimsup_{n\to\infty}\frac{\log\mathbb{P}(\mu_n\in A)}{\beta n^2}
    \leq \mathcal{E}_V(\mu_V)-\inf_{\mathrm{clo}(A)}\mathcal{E}_V
  \]
  where $\mathrm{int(A)}$ and $\mathrm{clo}(A)$ are the interior and closure
  of $A$ respectively. 
\end{theorem}

\begin{proof}[About the proof]
  The first proof of such a result dates back to \cite{MR1465640} and concerns
  one-dimensional log-gases. It is inspired by the work of Voiculescu on a
  Boltzmann point of view over free entropy and random matrices. Later
  contributions include \cite{MR1606719,MR2926763,MR3262506}. The approach
  developed in \cite{zbMATH07206384,MR3825945,zbMATH07133725} is very
  efficient.
\end{proof}

Theorem \ref{th:LDP} remains valid when $\beta=\beta_n$ provided that
\[
  \lim_{n\to\infty}n\beta_n=+\infty.
\]
This can be called the ``low temperature regime''. In the ``high temperature
regime'' $\beta=\beta_n$ with
\[
  \lim_{n\to\infty}n\beta_n=\kappa\in(0,+\infty),
\]
then Theorem \ref{th:LDP} remains valid provided that we replace
$\mathcal{E}_V$ by the new functional
\[
  \mathcal{E}+\frac{1}{\kappa}\mathrm{Entropy}(\cdot\mid\nu_{V,\kappa})
  =
  \mathcal{E}_V+\frac{1}{\kappa}\mathrm{Entropy}(\cdot\mid\mathrm{d}x)
  +c_{V,\kappa}
\]
where $\nu_{V,\kappa}$ has density proportional to $\mathrm{e}^{-\kappa V}$,
and where $\mathrm{Entropy}$ is the Kullback\,--\,Leibler divergence or
relative entropy. Note that $-\mathrm{Entropy}(\cdot\mid\mathrm{d}x)$ is by
definition the Boltzmann\,--\,Shannon entropy. We should also replace $\mu_V$
in Theorem \ref{th:LDP} by the minimizer of this new functional. This is also
known as the \emph{crossover regime}, interpolating between $\mu_V$ and
$\nu_{V,\kappa}$. Formally, if we turn off the interaction by taking $g=0$ and
if we take $\beta_n=\kappa/n$ then $P_n$ is the product probability measure
$\nu_{V,\kappa}^{\otimes n}$ and the large deviations principle becomes the
classical Sanov theorem associated to the law of large numbers for independent
random variables. The crossover regime is considered for instance in
\cite{MR1145596,MR1678526,zbMATH07133725,zbMATH07088025,armstrong-serfaty}.

\subsection{First order global asymptotics}

The (weak) convergence in $\mathcal{M}_1(\mathbb{R}^d)$ is metrized by the
bounded-Lipschitz distance defined by
\[
  \mathrm{d}_{\mathrm{BL}}(\mu,\nu)
  =\sup\Bigr\{\int f\mathrm{d}(\mu-\nu):
  \left\|f\right\|_\infty\leq1,\left\|f\right\|_{\mathrm{Lip}}\leq1\Bigr\}
\]
where the supremum runs over all measurable functions
$f:\mathbb{R}^d\to\mathbb{R}$ and where
\[
  \left\|f\right\|_\infty=\sup_x|f(x)|
  \quad\text{and}\quad
  \left\|f\right\|_{\mathrm{Lip}}=\sup_{x\neq y}\frac{|f(x)-f(y)|}{|x-y|}.
\]
Now for all $r\geq0$, by Theorem \ref{th:LDP} with
$A=A_r=\{\mu\in\mathcal{M}_1(\mathbb{R}^d):\mathrm{d}_{\mathrm{BL}}(\mu,\mu_V)\geq
r\}$, for $n$ large enough,
\begin{equation}\label{eq:LDPn}
  \mathrm{e}^{-c_r\beta n^2}\leq
  \mathbb{P}(\mathrm{d}_{\mathrm{BL}}(\mu_n,\mu_V)\geq r)
  \leq\mathrm{e}^{-C_r\beta n^2},
\end{equation}
where $c_r,C_r>0$ are constants depending on $A_r$ and $\mathcal{E}_V$ but not
on $n$. In particular, for all $\varepsilon>0$, 
\[
  \sum_n\mathbb{P}(\mathrm{d}_{\mathrm{BL}}(\mu_n,\mu_V)>\varepsilon)<\infty.
\]
By the Borel\,--\,Cantelli lemma, it follows that regardless of the way we choose
a common probability space to define the sequence of random vectors
${(X_n)}_n$, we have, almost surely,
\begin{equation}\label{eq:global}
  \lim_{n\to\infty}\mathrm{d}_{\mathrm{BL}}(\mu_n,\mu_V)=0.
\end{equation}
This is a sort of law of large numbers for our system of exchangeable
particles. They are not independent due to the Coulomb interaction, and the
information about the interaction remains in $\mu_V$. We refer to
\cite{MR3309890,berman,zbMATH07133725} for the relation to the notion of
``Gamma convergence''.

\subsection{Weakly confining versus strongly confining potential}

We could say that $V$ is \emph{weakly confining} when \eqref{eq:Vweak} holds,
and that $V$ is \emph{strongly confining} when \eqref{eq:Vstrong} holds. The
integrability condition \eqref{eq:Vint} may hold for weakly confining
potentials. An example of a two dimensional Coulomb gas with a weakly
confining potential is given by the Forrester\,--\,Krishnapur spherical
ensemble considered in the sequel, for which the equilibrium measure is not
compactly supported and is heavy-tailed.

\subsection{Concentration of measure}

The proof of Theorem \ref{th:LDP} can be adapted in order to provide
quantitative (meaning non-asymptotic) estimates for deviation probabilities.
Namely, for all $r\geq0$, 
\[
  \mathbb{P}(\mathrm{d}_{\mathrm{BL}}(\mu_n,\mu_V)\geq r)
  =\frac{1}{Z_n}\int_{\mathrm{d}_{\mathrm{BL}}(\mu_n,\mu_V)\geq r}
  \mathrm{e}^{-\beta n^2\mathcal{E}_V^{\neq}(\mu_{x_1,\ldots,x_n})}
  \mathrm{d}x_1\cdots\mathrm{d}x_n.
\]
Now if we could approximate $\mathcal{E}_V^{\neq}(\mu_{x_1,\ldots,x_n})$ with
$\mathcal{E}_V(\mu_{x_1,\ldots,x_n})$ and use an inequality of the form
\[
  \mathrm{d}_{\mathrm{BL}}(\mu,\mu_V)
  \leq c(\mathcal{E}_V(\mu)-\mathcal{E}_V(\mu_V)),
\]
and use a bound of the form 
\begin{equation}\label{eq:Zgeq}
  \log Z_n
  \geq
  n^2\mathcal{E}_V(\mu_V)+n(\beta\mathcal{E}(\mu_V)+c_V),
\end{equation}
then we would obtain a concentration of measure inequality of the form 
\[
  \mathbb{P}(\mathrm{d}(\mu_n,\mu_V)\geq r)
  \leq\mathrm{e}^{-cn^2r^2+o(n^2)}.
\]
for all $n$ and all $r\geq r_n$ for some threshold $r_n$. Actually the
quantity $\mathcal{E}_V(\mu_{x_1,\ldots,x_n})$ is infinite due to the atomic
nature of $\mu_{x_1,\ldots,x_n}$ and the method requires then a regularization
procedure. The details are in \cite{MR3820329}. The method, inspired by
\cite{MR3201924}, is related to \cite{MR3455593}. See also
\cite{MR3933036,zbMATH07107321,berman-conc,padilla-garza-concentration} for
other variations on this topic. Moreover we could replace the
bounded-Lipschitz distance by a Kantorovich\,--\,Wasserstein distance,
provided a growth assumption on $V$.

Such concentration inequalities around the equilibrium measure provide
typically an upper bound on the speed of the almost sure convergence. More
precisely if $r_n$ is such that
$\sum_n\mathrm{e}^{-cn^2r_n^2+o(n^2)}<\infty$, then by the
Borel\,--\,Cantelli lemma, we get that almost surely, for $n$ large enough,
\[
  \mathrm{d}_{\mathrm{BL}}(\mu_n,\mu_V) \leq r_n.
\]

%FIXME: give theorems

On the other hand, in the case of one-dimensional log-gases with strongly
convex potential $V$ such as the Gaussian unitary ensemble, another approach
is possible for concentration of measure, related to logarithmic Sobolev
inequalities, see for instance \cite{MR4175749} and references therein.

\section{Edge behavior}

We suppose in this section that $V$ is strongly confining, in particular the
equilibrium measure $\mu_V$ is compactly supported (Theorem
\ref{th:mustarprop}). The convergence \eqref{eq:global} holds in a weak sense,
which does not imply the convergence of the support. The most general result
about the convergence of the support is probably \cite{10.1214/21-EJP613}, and
appears as a refinement of \cite[Theorem 1.12]{MR3820329}. When $V$ is
rotationally invariant, this provides constants $c,r_*,p>0$ such that for all
$n$ and $r\geq r_*$,
\[
  \mathbb{P}\Bigr(\max_{1\leq k\leq n}|X_{n,k}|\geq r\Bigr)
  \leq\mathrm{e}^{-cnr^p}.
\]

%FIXME: give theorems

The fluctuation at the edge is a difficult subject which is well understood
for one-dimensional log-gases, for which it gives rise to Tracy\,--\,Widom
laws. For strongly confined rotationally invariant determinantal
two-dimensional Coulomb gases, it gives rise to Gumbel laws. An explicit
analysis of the Ginibre Coulomb gas is presented in the sequel (Theorem
\ref{th:ginibre-spectral-radius}), see also
\cite{MR3215627,MR3615091,doi:10.1063/1.5126724,MR4179777,10.1214/21-EJP613,raphael-david,david-edge,butez2021universality,chafai-garcia-zelada-jung-1d}
for more results in the same spirit.

\section{Global fluctuations and Gaussian free field}

Formally, from \eqref{eq:c2c} we could write
\[
  \mathcal{E}(\mu)
  =\frac{1}{2}\langle-c_d^{-1}\Delta U_\mu,U_\mu\rangle
  +\langle-c_d^{-1}\Delta V,U_\mu\rangle,
\]
and thus
\[
  \beta n^2\mathcal{E}(\mu)
  =\frac{1}{2}\langle-\beta c_d^{-1}\Delta U_{n\mu},U_{n\mu}\rangle
  +\langle -n\beta c_d^{-1}\Delta V,U_{n\mu}\rangle.
\]
In view of the Coulomb gas formula \eqref{eq:PnEneq}, this suggests to
interpret as $n\to\infty$ the random function $U_{n\mu_{x_1,\ldots,x_n}}$
under $P_n$ as a Gaussian with covariance operator $K=c_d(-\beta\Delta)^{-1}$.
Actually such an object is known as a \emph{Gaussian Free Field} (GFF). Next,
again from \eqref{eq:c2c}, this suggests to interpret formally as $n\to\infty$
the random measure
$n\mu_{x_1,\ldots,x_n}=-c_d^{-1}\Delta
U_{\mu_{x_1,\ldots,x_n}}=AU_{\mu_{x_1,\ldots,x_n}}$ under $P_n$ as a Gaussian
random measure with covariance operator
$A^2K=(-c_d^{-1}\Delta)^2(c_d(-\beta\Delta)^{-1})=-(\beta c_d)^{-1}\Delta$.
This argument would involve in principle a change of variable and a Jacobian,
that we do not consider here. This leads naturally to conjecture that for a
smooth enough test function $f:\mathbb{R}^d\to\mathbb{R}$,
\begin{equation}\label{eq:CLT}
  n\Bigr(\int f\mathrm{d}\mu_n
  -\mathbb{E}\int f\mathrm{d}\mu_n\Bigr)
  =\sum_{k=1}^nf(X_{n,k})-\mathbb{E}(f(X_{n,k}))
  \underset{n\to\infty}{\overset{\mathrm{d}}{\longrightarrow}}
  \mathcal{N}\Bigr(0,\frac{1}{\beta c_d}\int|\nabla f|^2\mathrm{d}x\Bigr).
\end{equation}
This can be seen as the ``central limit theorem'' statement associated to the
``law of large numbers'' statement \eqref{eq:global}. The limiting variance
could be perturbed by edge effects depending on the relative position of the
support and regularity of $f$ and $\mu_V$. We could have also an additional
bias correction.

% FIXME: give theorems

The Coulomb interaction together with the confinement produces a rigidity of
the global configuration and reduces the variance of linear statistics. Indeed
\eqref{eq:CLT} comes with an $n$ scaling that differs from the usual
$\sqrt{n}$ scaling for independent random variables (no interaction).

The covariance of the limiting Gaussian in \eqref{eq:CLT} is easily guessed
from the Hessian at the minimizer of the rate function in the large deviations
principle of Theorem \ref{th:LDP}. This CLT\,--\,LDP link is well known. The
GFF is an example of a log-correlated Gaussian field \cite{zbMATH06741334}, a
fashionable subject.

A statement similar to \eqref{eq:CLT} is proved rigorously in \cite{MR2361453}
for the complex Ginibre ensemble by using its exact solvability (determinantal
structure). See also \cite{MR3342661,MR2817648}. Extensions to non-exactly
solvable two-dimensional Coulomb gases are considered in
\cite{MR4063572,MR3788208,serfaty-clt,leble-zeitouni}.
% Determiantal nature, plus cumulant to compute variance %
% Costin-Lebowitz \cite{MR3155254} %
% Jancovici-Lebowitz-Manifica \cite{MR1239571}

For one-dimensional log-gases emerging from random matrix theory, central
limit theorems such as \eqref{eq:CLT} were established using the Laplace
transform and ``loop equations'' in \cite{MR1487983}. See also
\cite{MR2808038,MR3668648,MR3885548,MR4211032} and references therein for
extensions and generalizations.

\section{Aspects of general exact solvability}

Theorem \ref{th:exactgeneral} is taken from \cite{blog-vw} and
\cite{MR4175749} (see also \cite{zbMATH07293730}).

\begin{theorem}[Exact distributions for special linear statistics of general gases]\label{th:exactgeneral}
  Let $X=(X_1,\ldots,X_n)$ be a random vector of $(\mathbb{R}^d)^n$,
  $n,d\geq1$, with density proportional to
  \[
    \mathrm{e}^{-\sum_i V(x_i)}\prod_{i<j}W(x_i-x_j)
  \]
  where $V:\mathbb{R}^d\to[0,+\infty]$ and $W:\mathbb{R}^d\to[0,+\infty]$ are
  measurable.
  \begin{itemize}
  \item If $V$ and $W$ are homogeneous in the sense that for some $a,b\geq0$,
    and for all $\lambda\geq0$ and $x\in\mathbb{R}^d$,
    \[
      V(\lambda x)=\lambda^a V(x)
      \quad\text{and}\quad
      W(\lambda x)=\lambda^bW(x),
    \]
    then
    \[
      V(X_1)+\cdots+V(X_n)
      \sim\mathrm{Gamma}\Bigr(\frac{nd}{a}+\frac{n(n-1)b}{2a},1\Bigr).
    \]
  \item If $V=\gamma\left|\cdot\right|^2$ for some $\gamma>0$ then
    \[
      X_1+\cdots+X_n\sim\mathcal{N}\Bigr(0,\frac{n}{2\gamma}I_d\Bigr),
    \]
    and moreover the orthogonal projection $\pi$ on the subspace
    $\{(z,\ldots,z):z\in\mathbb{R}^d\}$ of $(\mathbb{R}^d)^n$ satisfies
    $\pi(X)=\frac{X_1+\cdots+X_n}{n}(1,\ldots,1)$, and furthermore $\pi(X)$
    and $\pi^\perp(X)=X-\pi(X)$ are independent.
  \end{itemize}
\end{theorem}

\begin{proof}[Proof of Theorem \ref{th:exactgeneral}]
  Recall that linear change of variable is valid for integrals of measurable
  functions. \emph{First formula.} For all $\theta>0$, we have, with the
  substitution $x_i=\bigr(\frac{1}{1+\theta}\bigr)^{1/a}y_i$,
  \[
    \int_{(\mathbb{R}^d)^n}
    \mathrm{e}^{-\theta\sum_iV(x_i)}
    \mathrm{e}^{-\sum_iV(x_i)}
    \prod_{i<j}W(x_i-x_j)\mathrm{d}x
    =\Bigr(\frac{1}{1+\theta}\Bigr)^{\frac{nd}{a}+\frac{(n^2-n)}{2}\frac{b}{2a}}
    \int_{(\mathbb{R}^d)^n}\mathrm{e}^{-\sum_iV(y_i)}\prod_{i<j}W(y_i-y_j)\mathrm{d}y.
  \]
  We recognize the Laplace transform of 
  $\mathrm{Gamma}\bigr(\frac{nd}{a}+\beta\frac{n(n-1)b}{2a},1\bigr)$, namely
  \[
    \int_0^\infty\mathrm{e}^{-\theta x}x^{\alpha-1}\mathrm{e}^{-\lambda x}\mathrm{d}x
    =\int_0^\infty x^{\alpha-1}\mathrm{e}^{-(\lambda+\theta)x}\mathrm{d}x
    =\Bigr(\frac{\lambda}{\lambda+\theta}\Bigr)^\alpha\frac{\Gamma(\alpha)}{\lambda^\alpha},
  \]
  therefore 
  $\sum_iV(X_i)\sim\mathrm{Gamma}\bigr(\frac{nd}{a}+\frac{n(n-1)b}{2a},1\bigr)$.

  \emph{Second formula.} For all $\theta\in\mathbb{R}^d$, we have, with the
  substitution $y_i=x_i+\frac{1}{2\gamma}\theta$ (a translation or shift), 
  \[
    \int_{(\mathbb{R}^d)^n}
    \mathrm{e}^{-\theta\cdot\sum_i x_i}
    \mathrm{e}^{-\sum_iV(x_i)}    
    \prod_{i<j}W(x_i-x_j)\mathrm{d}x
    =
    \mathrm{e}^{\frac{n}{2\gamma}\frac{|\theta|^2}{2}}
    \int_{(\mathbb{R}^d)^n}\mathrm{e}^{-\sum_iV(y_i)}\prod_{i<j}W(y_i-y_j)\mathrm{d}y,
  \]
  and we recognize the Laplace transform of the Gaussian law
  $\mathcal{N}\bigr(0,\frac{n}{2\gamma}I_d\bigr)$. Finally the properties
  related to $\pi(X)$ follow from the quadratic nature of $V$ and the shift
  invariance of $W$, and correspond to a factorization of the law of $X$,
  namely, denoting $\pi^\perp(x)=x-\pi(x)$ and using
  $|x|^2=|\pi(x)|^2+|\pi^\perp(x)|^2$ (Pythagoras theorem) and
  $x_i-x_j=\pi^\perp(x)_i-\pi^\perp(x)_j$ (from the definition of $\pi$), we
  get
  \[
    \mathrm{e}^{-\sum_i V(x_i)}\prod_{i<j}W(x_i-x_j)
    =\mathrm{e}^{-\gamma|\pi(x)|^2}
    \times
    \mathrm{e}^{-\gamma|\pi^\perp(x)|^2}\prod_{i<j}W(\pi^\perp(x)_i-\pi^\perp(x)_j).
  \]
  This provides the independence of $\pi(X)$ and $\pi^\perp(X)$ as well as the
  fact that $\pi(X)\sim\mathcal{N}(0,\frac{1}{2\gamma}I_d)$.
\end{proof}

\begin{corollary}[Exact laws for beta-Ginibre gases]\label{co:beta}
  Let us consider $X_n=(X_{n,1},\ldots,X_{n,n})\sim P_n$ where $P_n$ is as in
  \eqref{eq:beta} with $\beta>0$ and $V=\frac{1}{2}\left|\cdot\right|^2$. In
  other words, the density of $P_n$ in $\mathbb{C}^n$ is given by
  \[
    (z_1,\ldots,z_n)\in\mathbb{C}^n
    \mapsto\frac{\mathrm{e}^{-n\frac{\beta}{2}(|z_1|^2+\cdots+|z_n|^2)}}{Z_n}\prod_{i<j}|z_i-z_j|^\beta.
  \]
  Then 
  \[
    X_{n,1}+\cdots+X_{n,n}\sim\mathcal{N}\Bigr(0,\frac{I_2}{\beta}\Bigr)
    \quad\text{and}\quad
    |X_n|^2=|X_{n,1}|^2+\cdots+|X_{n,n}|^2
    \sim\mathrm{Gamma}\Bigr(n+\beta\frac{n(n-1)}{4},\beta\frac{n}{2}\Bigr),
  \]
  and in particular
  \[
    \mathbb{E}(|X_{n,1}+\cdots+X_{n,n}|^2)=\frac{2}{\beta}
    \quad\text{and}\quad
    \mathbb{E}(|X_n|^2)
    =\mathbb{E}(|X_{n,1}|^2+\cdots+|X_{n,n}|^2)
    =\frac{2}{\beta}+\frac{n-1}{2}.
  \]
\end{corollary}

When $\beta=2$ we recover the Ginibre gas \eqref{eq:ginibre}. Beyond this
case, and up to our knowledge, it seems that there is no useful matrix model
with independent entries for which the spectrum follows this $\beta$ gas.

With $\beta=\frac{1}{n}$, we get as $n\to\infty$ that the variance of the
Gauss\,--\,Ginibre crossover is $2+\frac{1}{2}=\frac{5}{2}$.

\begin{proof}[Proof of Corollary \ref{co:beta}]
  It suffices to use Theorem \ref{th:exactgeneral} with $d=2$,
  $V=n\frac{\beta}{2}\left|\cdot\right|^2$, $W=\left|\cdot\right|^\beta$, for
  which $a=2$ and $b=\beta$, and the scaling property
  $\sigma Z\sim\mathrm{Gamma}\bigr(\alpha,\frac{\lambda}{\sigma}\bigr)$ when
  $Z\sim\mathrm{Gamma}(\alpha,\lambda)$, for any $\sigma>0$.

  Note that in the determinantal case $\beta=2$, Theorem \ref{th:kostlan}
  gives that $n|X|^2\sim\mathrm{Gamma}(1+2+\cdots+n,1)$ since it has the law
  of a sum of $n$ independent random variables of law
  $\mathrm{Gamma}(1,1),\ldots,\mathrm{Gamma}(n,1)$.
\end{proof}

% FIXME: Theorem \ref{th:beta} could also be used to compute logarithmic
% moments for the Spherical ensemble $V(x)=(n+1)\log(1+\left|\cdot\right|)$
% and same $W$.

\begin{remark}[Real case]
  For all $\beta>0$, $n\geq2$, let $P_n$ be the law on $\mathbb{R}^n$ with
  density 
  \[
    (x_1,\ldots,x_n)\in\mathbb{R}^n\mapsto
    \frac{\mathrm{e}^{-n\frac{\beta}{4}(x_1^2+\cdots+x_n^2)}}{Z_n}\prod_{i<j}|x_i-x_j|^\beta.
  \]
  The normalization $Z_n$ can be explicitly computed via a Mehta\,--\,Selberg
  integral \cite{zbMATH05348712}. It is a quadratically confined
  one-dimensional log-gas known as the real beta Hermite gas. The case
  $\beta=2$ corresponds to GUE. If $X_n=(X_{n,1},\ldots,X_{n,n})\sim P_n$,
  then the proof of Theorem \ref{co:beta} provides
  \[
    X_{n,1}+\cdots+X_{n,n}=\mathcal{N}\left(0,\frac{2}{\beta}\right)
    \quad\text{and}\quad
    X_{n,1}^2+\cdots+X_{n,n}^2
    \sim\mathrm{Gamma}\left(\frac{n}{2}+\frac{\beta n(n-1)}{4},\frac{\beta
        n}{4}\right),
  \]
  and in particular,
  \[
    \mathbb{E}((X_{n,1}+\cdots+X_{n,n})^2)=\frac{2}{\beta}
    \quad\text{and}\quad
    \mathbb{E}(X_{n,1}^2+\cdots+X_{n,n}^2)=\frac{2}{\beta}+n-1.
  \]
  Alternatively, these formulas can also be derived by using the tridiagonal
  random matrix model of Dumitriu and Edelman \cite{MR1936554} valid for all
  real beta Hermite gases, see for instance \cite{MR4175749}.
\end{remark}

\begin{remark}[Langevin dynamics]
  The Boltzmann\,--\,Gibbs measure $P_n$ defined in \eqref{eq:P} is the
  invariant law of the Kolmogorov diffusion process ${(X_t)}_{t\geq0}$
  solution of the stochastic differential equation
  \begin{equation}\label{eq:SDE}
    \mathrm{d}X_t
    =\sqrt{2\frac{\alpha}{\beta}}\mathrm{d}B_t-\alpha\nabla E_n(X_t)\mathrm{d}t
  \end{equation}
  where ${(B_t)}_{t\geq0}$ is a standard Brownian motion on
  $(\mathbb{R}^d)^n$, and where $\alpha>0$ is an arbitrary parameter which
  corresponds to a deterministic time change. The infinitesimal generator of
  the associated Markov semi-group is the second order linear differential
  operator without constant term
  \begin{equation}\label{eq:L}
    L=\alpha\Bigr(\frac{1}{\beta}\Delta-\nabla E_n\cdot\nabla\Bigr).
  \end{equation}
  The most standard parametrizations are $\alpha=1$, which allows to interpret
  $1/\beta$ as the temperature of the Brownian part, and $\alpha=\beta$. See
  for instance \cite{MR3434251,MR2352327,MR3155209}. Since $E_n$ is a two-body
  interaction energy, the operator $L$ can be seen as a mean-field particle
  approximation of a McKean\,--\,Vlasov dynamics, see for instance
  \cite{MR3847984,MR4158670}. % \cite{guillin2021uniform}\cite{rosenzweig2021globalintime}
  The singularity of $g$ makes non-obvious the well-posedness or absence of
  explosion of \eqref{eq:SDE}, and we refer to
  \cite{MR1217451,MR2760897,MR3699468} for one-dimensional log-gases, and to
  \cite{MR3847984} for the (two-dimensional) beta-Ginibre gas. See also
  \cite{zbMATH07202715} and \cite{zbMATH07088025} for more recent results on
  such dynamics. Historically \eqref{eq:SDE} emerges as the description of the
  dynamics of the spectrum of Hermitian Ornstein\,--\,Uhlenbeck processes, and
  is nowadays called a Dyson process, named after \cite{MR0148397}. We say
  that \eqref{eq:SDE} is a gradient dynamics because the drift is the gradient
  of a function. From the point of view of statistical physics a stochastic
  differential equation such as \eqref{eq:SDE} is also known as an overdamped
  Langevin dynamics, which is a degenerate version of the true (kinetic
  under-damped) Langevin dynamics, see for instance \cite{MR3911782}. Langevin
  dynamics can be used for the numerical simulation of $P_n$, see for instance
  \cite{MR3911782,zbMATH07293730} and references therein. Dynamics such as
  \eqref{eq:SDE} can be used as an interpolation device between $X_0$ and
  $X_\infty\sim P_n$, possibly by using conservation laws related to
  eigenfunctions. In this spirit, and following \cite{MR3847984,MR4175749},
  we could prove Corollary \ref{co:beta} by using the fact that $\sum_ix_i$
  and $\sum_i|x_i|^2$ are essentially eigenfunctions of \eqref{eq:L},
  producing Ornstein\,--\,Uhlenbeck and Cox\,--\,Ingersoll\,--\,Ross processes
  for which the targeted Gaussian and Gamma laws are invariant.
\end{remark}

\section{Exactly solvable two-dimensional gases from random matrix theory}

The spectrum of several random matrix models are gases in dimension
$d\in\{1,2\}$ with $W=g$. The cases $\beta\in\{1,2,4\}$ play often a special
role related to algebra. We refer to \cite{MR2129906,MR2168344,MR2641363} for
more details on the zoology of random matrices. It is natural to ask if there
exists a random matrix model with independent entries for which the spectrum
is distributed according to the beta gas \eqref{eq:beta}. Up to our knowledge,
the answer is negative for \eqref{eq:beta} in general but positive for the
Ginibre gas \eqref{eq:ginibre}.

\subsection{Ginibre model}\label{se:ginibre}

A (complex) Ginibre random matrix $\mathbf{M}$ is an $n\times n$ complex
matrix such that 
\begin{equation}\label{eq:ginibre-matrix}
  \Bigr\{\Re\mathbf{M}_{i,j},\Im\mathbf{M}_{i,j}:1\leq i,j\leq n\Bigr\}
\end{equation}
are independent and Gaussian random variables of law
$\mathcal{N}(0,\frac{1}{2})$. In other words 
the complex random variables $\{\mathbf{M}_{i,j}:1\leq i,j\leq n\}$ are
independent and Gaussian of law $\mathcal{N}(0,\frac{1}{2}I_2)$. Note that
$\mathbb{E}(|\mathbf{M}_{i,j}|^2)=1$.

Let $(\boldsymbol{\lambda}_1,\ldots,\boldsymbol{\lambda}_n)$ be the
eigenvalues of $\mathbf{M}$ seen as an exchangeable random vector of
$\mathbb{C}^n$. This means that we randomize the numbering of the eigenvalues
with an independent uniform random permutation of $\{1,\ldots,n\}$.
Equivalently, this corresponds to consider the random multi-set encoding the
spectrum, keeping by this way the possible multiplicities but discarding the
numbering of the eigenvalues.

\begin{theorem}[From the Ginibre random matrix to the Ginibre gas]\label{th:ginibre} %
  The exchangeable random vector
  \[
    \Bigr(\frac{\boldsymbol{\lambda}_1}{\sqrt{n}},\ldots,\frac{\boldsymbol{\lambda}_n}{\sqrt{n}}\Bigr)
  \]
  is distributed according to the Ginibre gas \eqref{eq:ginibre}. In other
  words $(\boldsymbol{\lambda}_1,\ldots,\boldsymbol{\lambda}_n)$ has density
  $\varphi_n$ as in \eqref{eq:ginibre}.
\end{theorem}

\begin{proof}[Idea of the proof]
  The set of $n\times n$ complex matrices with multiple eigenvalues has zero
  Lebesgue measure. Since the law of $\mathbf{M}$ is absolutely continuous, it
  follows that almost surely $\mathbf{M}$ is diagonalizable with distinct
  eigenvalues. The density is proportional to ($M^*=\overline{M}^\top$ is the
  conjugate-transpose of $M$)
  \[
    M\mapsto\mathrm{e}^{-\sum_{i,j=1}^n|M_{i,j}|^2}=\mathrm{e}^{-\mathrm{Trace}(MM^*)}.
  \]
  In order to compute the law of the spectrum of $\mathbf{M}$, an idea is to
  use for instance the Schur unitary decomposition as a change of variable.
  Namely, if $M$ is diagonalizable, then the Schur decomposition is the matrix
  factorization $M=U(D+N)U^*$ where $U$ is unitary, $D$ is diagonal, and $N$
  is upper triangular with null diagonal (nilpotent). The matrix $D$ carries
  the eigenvalues of $M$. We have the decoupling
  \[
    \mathrm{Trace}(MM^*)=\mathrm{Trace}(DD^*)+\mathrm{Trace}(NN^*).
  \]
  This allows to integrate out $(N,U)$ in the density and to get that the law
  of the eigenvalues of $\mathbf{M}$ is given by \eqref{eq:ginibre}. The term
  $\prod_{i<j}|x_i-x_j|^2$ is the modulus of the determinant of the Jacobian
  of the change of variable. We obtain that for every symmetric bounded (or
  positive) measurable function $F:\mathbb{C}^n\to\mathbb{R}$,
  \[  
    \mathbb{E}[F(\boldsymbol{\lambda}_1,\ldots,\boldsymbol{\lambda}_n)]  
    =\int_{\mathbb{C}^n}\!F(z_1,\ldots,z_n)\varphi_n(z_1,\ldots,z_n)
    \mathrm{d}z_1\cdots\mathrm{d}z_n
  \]
  where $\mathrm{d}z_1\cdots\mathrm{d}z_n$ stands for the Lebesgue measure on
  $\mathbb{C}^n=\mathbb{R}^{2n}$. The result goes back to \cite{MR0173726}.
  The scheme of proof that we follow here can be found in \cite{MR2932638},
  see also \cite[Ch.~15]{MR2129906,MR2641363}.
\end{proof}

\begin{remark}[Immediate properties of Ginibre random matrices]\ %
  \begin{itemize}
  \item Since the law of $\mathbf{M}$ is absolutely continuous, almost surely
    $\mathbf{M}\mathbf{M}^*\neq\mathbf{M}^*\mathbf{M}$ (non-normality);
  \item By the law of large numbers, almost surely, as $n\to\infty$,
    $\frac{1}{\sqrt{n}}\mathbf{M}$ has orthonormal rows/columns;
  \item The law of $\mathbf{M}$ is bi-unitary invariant: if $U$ and $V$ are
    unitary then $U\mathbf{M}V$ and $\mathbf{M}$ have same law;
  \item The Hermitian random matrices
    $\frac{1}{\sqrt{2}}(\mathbf{M}+\mathbf{M}^*)$ and
    $\frac{1}{\sqrt{2}\mathrm{i}}(\mathbf{M}-\mathbf{M}^*)$, the matrix real
    and imaginary parts of $\mathbf{M}$, are independent and belong to the
    Gaussian Unitary Ensemble (GUE): their density is proportional to
    $H\mapsto\mathrm{e}^{-\frac{1}{2}\mathrm{Trace}(H^2)}$; Conversely, if
    $\mathbf{H}_1$ and $\mathbf{H}_2$ are independent copies of the Gaussian
    Unitary Ensemble then the random matrices
    $\frac{1}{\sqrt{2}}(\mathbf{H}_1+\mathrm{i}\mathbf{H}_2)$ and $\mathbf{M}$
    have same law.
  \end{itemize}
\end{remark}

The exact solvability of the Ginibre gas \eqref{eq:ginibre} is largely due to
a determinantal structure studied below, itself related to the fact that
$\beta=2$ and $W=g$. More precisely, first of all, from \eqref{eq:ginibre}
we have
\begin{equation}\label{eq:vphin}
  \varphi_n(z_1,\ldots,z_n)%
  =\frac{\prod_{k=1}^n\gamma(z_k)}{\prod_{k=1}^nk!}\prod_{i<j}|z_i-z_j|^2
\end{equation}
where $\gamma$ is the density of $\mathcal{N}(0,\frac{1}{2}I_2)$ given for all
$z\in\mathbb{C}$ by
\[
	\gamma(z)=\frac{\mathrm{e}^{-|z|^2}}{\pi}.
\]
  
\begin{theorem}[Determinantal structure and
  marginals]\label{th:ginibre-k-points} For all $n\geq1$ and $(z_1,\ldots,z_n)\in\mathbb{C}^n$,
  \[
    \varphi_n(z_1,\ldots,z_n) %
    %=\varphi_{n,n}(x_1,\ldots,x_n) %
    =\frac{1}{n!}\det\left[K_n(z_i,z_j)\right]_{1\leq i,j\leq n}
  \]
  where the kernel $K_n$ is given for all $z,w\in\mathbb{C}$ by
  \[  
    K_n(z,w)  
    =\sqrt{\gamma(z)\gamma(w)}\sum_{\ell=0}^{n-1}\frac{(z\overline{w})^\ell}{\ell!}.
  \]  
  More generally, for all $1\leq k\leq n$, the marginal density
  \[
    (z_1,\ldots,z_k)\in\mathbb{C}^k\mapsto
    \varphi_{n,k}(z_1,\ldots,z_k) =  
    \int_{\mathbb{C}^{n-k}}\varphi_n(z_1,\ldots,z_n)\mathrm{d}z_{k+1}\cdots\mathrm{d}z_n  
  \]  
  satisfies, for all $(z_1,\ldots,z_k)\in\mathbb{C}^k$,
  \[  
    \varphi_{n,k}(z_1,\ldots,z_k) =   
    \frac{(n-k)!}{n!}\det\left[K_n(z_i,z_j)\right]_{1\leq i,j\leq k}.
  \]
  In particular for $k=n$ we get $\varphi_{n,n}=\varphi_n$, while for $k=1$ we
  get, for all $z\in\mathbb{C}$,
  \[
    \varphi_{n,1}(z)=\frac{\gamma(z)}{n}\sum_{\ell=0}^{n-1}\frac{|z|^{2\ell}}{\ell!}.
  \]
\end{theorem}

We say that the spectrum of $\mathbf{M}$ is a Gaussian determinantal point
process, see \cite[Ch.~4]{MR2552864}.

The ``$k$-point correlation'' is
$R_{n,k}(z_1,\ldots,z_k)=\frac{n!}{(n-k)!}\varphi_{n,k}(z_1,\ldots,z_k)=\det[K_n(z_i,z_j)]_{1\leq
  i,j\leq k}$.

\begin{proof}[Idea of proof]
  Following for instance \cite[Sec.~5.2 and Ch.~15]{MR2129906}, we get,
  starting with \eqref{eq:vphin},
  \begin{align*}
    \varphi_n(z_1,\ldots,z_n)
    &=\frac{\prod_{k=1}^n\gamma(z_k)}{\prod_{k=1}^nk!}
      \prod_{1\leq i{<}j\leq n}(z_i-z_j)\prod_{1\leq i{<}j\leq n}\overline{(z_i-z_j)}\\
    &=\frac{\prod_{k=1}^n\gamma(x_k)}{n!}
      \det\Bigr[\frac{z_j^{i-1}}{\sqrt{(i-1)!}}\Bigr]_{1\leq i,j\leq n}\det\Bigr[\frac{\overline{z_j}^{i-1}}{\sqrt{(i-1)}!}\Bigr]_{1\leq  
      i,j\leq n}\\
    &=\frac{1}{n!}\det\Bigr[K_n(z_i,z_j)\Bigr]_{1\leq i,j\leq n}.      
  \end{align*}
  On the other hand, the orthogonality of
  $\{\frac{z^\ell}{\sqrt{\ell!}}:0\leq\ell\leq n-1\}$ in
  $\mathrm{L}^2(\mathbb{C},\gamma)$ gives the identities
  \[
    \int_{\mathbb{C}}K_n(x,x)\mathrm{d}x=n
    \quad\text{and}\quad
    \int_{\mathbb{C}}K_n(x,y)K_n(y,z)\mathrm{d}y
    =K_n(x,z),\quad x,z\in\mathbb{C}.
  \]
  Finally the formula for $\varphi_{n,k}$ follows by expanding the determinant
  in $\varphi_n$ and using these identities.
\end{proof}  

\begin{theorem}[Mean circular Law]  
  \label{th:ginibre-circular-weak}
  Let $\boldsymbol{\lambda}_1,\ldots,\boldsymbol{\lambda}_n$ be as in Theorem
  \ref{th:ginibre} and let us define 
  \[
    \mu_n=\frac{1}{n}\sum_{k=1}^n\delta_{\frac{\boldsymbol{\lambda}_k}{\sqrt{n}}}.
  \]
  Let $\mu_\infty$ be the uniform distribution on the unit disc
  $\{z\in\mathbb{C}:|z|\leq1\}$ with density
  $z\in\mathbb{C}\mapsto\frac{\mathbf{1}_{|z|\leq1}}{\pi}$. Then
  \[
    \mathbb{E}\mu_n\weak\mu_\infty.
  \]
\end{theorem}  
  
\begin{proof}%[Idea of the proof]
  Let $\varphi_{n,1}$ be as in Theorem \ref{th:ginibre-k-points}. For
  all $f\in\mathcal{C}_b(\mathbb{C},\mathbb{R})$, we have, using Theorem
  \ref{th:ginibre},
  \[
    \mathbb{E}\int f\mathrm{d}\mu_n
    =\frac{1}{n}\sum_{k=1}^n\int_{\mathbb{C}^n}f\Bigr(\frac{z_k}{\sqrt{n}}\Bigr)\varphi_n(z_1,\ldots,z_n)\mathrm{d}z_1\cdots\mathrm{d}z_n
    =n\int_{\mathbb{C}}f(z)\varphi_{n,1}(\sqrt{n}z)\mathrm{d}z.
  \]
  Thus $\mathbb{E}\mu_n$ has density
  $n\varphi_{n,1}(\sqrt{n}\bullet)$. By Theorem
  \ref{th:ginibre-k-points} and Lemma \ref{le:exp}, if
  $K\subset\{z\in\mathbb{C}:|z|\neq1\}$ is compact,
  \[
    \lim_{n\to\infty}\sup_{z\in K}
    \Bigr|n\varphi_{n,1}(\sqrt{n}z)-\frac{\mathbf{1}_{|z|\leq1}}{\pi}\Bigr| %
    = \frac{1}{\pi}\lim_{n\to\infty}\sup_{z\in K}
    \Bigr|\mathrm{e}^{-n|z|^2}\mathrm{e}_n(n|z|^2)-\mathbf{1}_{|z|\leq1}\Bigr|=0.
  \]
  It follows then by dominated convergence that $\mathbb{E}\mu_n\weak\mu_\infty$.
\end{proof}

\begin{lemma}[Exponential series]\label{le:exp}
  For every $n\geq1$ and $z\in\mathbb{C}$,
  \[
  |\mathrm{e}_n(nz)-\mathrm{e}^{nz}\mathbf{1}_{|z|\leq1}|\leq r_n(z)
  \]
  where $\mathrm{e}_n(z)=\sum_{\ell=0}^{n-1}\frac{z^\ell}{\ell!}$ is the
  truncated exponential series and
  \[
  r_n(z)=
  \frac{\mathrm{e}^n}{\sqrt{2\pi n}}|z|^n%
  \Bigr(\frac{n+1}{n(1-|z|)+1}\mathbf{1}_{|z|\leq1}%
  +\frac{n}{n(|z|-1)+1}\mathbf{1}_{|z|>1}\Bigr).
  \]
\end{lemma}

\begin{proof}[Proof of Lemma \ref{le:exp}]
  As in Mehta \cite[Ch.~15]{MR2129906}, for every $n\geq1$, $z\in\mathbb{C}$, if
  $|z|\leq n$ then
  \[
  \Bigr|\mathrm{e}^{z}-\mathrm{e}_n(z)\Bigr|
  =\Bigr|\sum_{\ell=n}^\infty\frac{z^\ell}{\ell!}\Bigr|
  \leq\frac{|z|^n}{n!}\sum_{\ell=0}^\infty\frac{|z|^\ell}{(n+1)^\ell}
  =\frac{|z|^n}{n!}\frac{n+1}{n+1-|z|},
  \]
  while if $|z|>n$ then
  \[
  |\mathrm{e}_n(z)|
  \leq \sum_{\ell=0}^{n-1}\frac{|z|^\ell}{\ell!}
  \leq \frac{|z|^{n-1}}{(n-1)!}\sum_{\ell=0}^{n-1}\frac{(n-1)^{\ell}}{|z|^\ell}
  \leq\frac{|z|^{n-1}}{(n-1)!}\frac{|z|}{|z|-n+1}.
  \]
  Therefore, for every $n\geq1$ and $z\in\mathbb{C}$,
  \[
  |\mathrm{e}_n(nz)-\mathrm{e}^{nz}\mathbf{1}_{|z|\leq1}|
  \leq
  \frac{n^n}{n!}\Bigr(|z|^n\frac{n+1}{n+1-|nz|}\mathbf{1}_{|z|\leq1}
    +|z|^{n-1}\frac{|nz|}{|nz|-n+1}\mathbf{1}_{|z|>1}\Bigr).
  \]
  It remains to use the Stirling bound $\sqrt{2\pi n}n^n\leq n!\mathrm{e}^n$
  to get the first result. 
\end{proof}

\begin{remark}[Probabilistic view]\label{rk:poi}
  There is a probabilistic interpretation of Lemma \ref{le:exp}. For all $z\in\mathbb{C}$,
  \[
    \lim_{n\to\infty}n\varphi_{n,1}(\sqrt{n}z)
    =\frac{1}{\pi}\Bigr(\mathbf{1}_{|z|<1} + \frac{1}{2}  \mathbf{1}_{|z|=1}\Bigr).
  \]
  Namely, by rotational invariance, it suffices to consider the case
  $z = r >0$. Next, if $Y_1, \dots, Y_n$ are independent and identically
  distributed random variables following the Poisson law of mean
  $r^2$, then
  \[
    \mathrm{e}^{-n r^2}\mathrm{e}_n(nr^2) %
    =\mathbb{P}(Y_1+\cdots+Y_n<n) %
    =\mathbb{P}\Bigr(\frac{Y_1+\cdots+Y_n}{n}<1\Bigr).
  \]
  Now $\lim_{n\to\infty}\frac{Y_1 + \dots + Y_n}{n}=r^2$ almost surely by the
  strong law of large numbers, and thus the probability in the right-hand side
  above tends as $n\to\infty$ to $0$ if $r>1$ and to $1$ if $r<1$. In other
  words, for all $r\neq1$,
  \[
    \lim_{n\to\infty}\mathrm{e}^{-n r^2}\mathrm{e}_n(nr^2) = \mathbf{1}_{r < 1}.
  \]
  It remains to note that for $r=1$ by the central limit theorem we get
  \[
    \mathbb{P}\Bigr(\frac{Y_1+\cdots+Y_n}{n}<1\Bigr)
    = \mathbb{P}\Bigr(\frac{Y_1+\cdots+Y_n-n}{\sqrt{n}}<0\Bigr)
    \underset{n\to\infty}{\longrightarrow}\frac{1}{2}.
  \]
  % As shown in the proof of Lemma \ref{le:exp}, the convergence of
  % $n\varphi_{n,1}(\sqrt{n}z)$ holds uniformly on over $z$ on compact sets
  % outside the unit circle $\vert z \vert = 1$. However this convergence
  % cannot
  % hold uniformly on arbitrary compact sets of $\mathbb{C}$ since the
  % pointwise
  % limit is not continuous on the unit circle.
\end{remark}

\begin{remark}[Incomplete gamma function]
  It is well known that the Gamma and the Poisson laws are connected. Namely,
  if $X\sim\mathrm{Gamma}(n,\lambda)$ with $n\geq1$ and $\lambda>0$ and
  $Y\sim\mathrm{Poisson}(r)$ with $r>0$ then
  \[
    \mathbb{P}(X\geq\lambda r)
    =\frac{1}{(n-1)!}\int_r^\infty x^{n-1}\mathrm{e}^{-x}\mathrm{d}x
    =\mathrm{e}^{-r}\sum_{\ell=0}^{n-1}\frac{r^\ell}{\ell!}
    =\mathbb{P}(Y\geq n).
  \]
  Also we could use Gamma random variables instead of Poisson random variables
  in Remark \ref{rk:poi}. Note also that the integral in the middle of the
  formula above is the incomplete Gamma function $\Gamma(n,r)$. This allows to
  benefit from the asymptotic analysis of this special function, see
  \cite{MR2932638} and references therein.
\end{remark}

% The sequence $(H_k)_{k\in\mathbb{N}}$ forms an orthonormal basis (orthogonal
% polynomials) of square integrable analytic functions on $\mathbb{C}$ for the standard
% Gaussian on $\mathbb{C}$. The uniform law on the unit disc is the law of
% $\sqrt{V}\mathrm{e}^{2i\pi W}$ where $V$ and $W$ are i.i.d.\ uniform random variables
% on the interval $[0,1]$. This can be used to interpolate between complex
% Ginibre and GUE via Girko's elliptic laws, see
% \cite{MR2446909,MR2288065,MR2594353}.

\begin{theorem}[Strong circular law]\label{th:ginibre-circular-strong}
  With the notations of Theorem \ref{th:ginibre-circular-weak}, almost
  surely,
  \[
    \mu_n\weak\mu_\infty.
  \]
\end{theorem}  

Note that this convergence holds regardless of the way we define
the random matrices on the same probability space when $n$ varies. This is an
instance of the concept of \emph{complete convergence}, see \cite{MR1632875}.

\begin{proof}[Idea of the proof]
  The argument, due to Jack Silverstein, is in \cite{MR841088}. It is similar
  to the quick proof of the strong law of large numbers for independent random
  variables with bounded fourth moment. It suffices to establish the result
  for an arbitrary compactly supported
  $f\in\mathcal{C}_b(\mathbb{C},\mathbb{R})$. Let us define
  \[  
  S_n=\int_{\mathbb{C}}\!f\,\mathrm{d}\mu_n
  \quad\text{and}\quad  
  S_\infty=\frac{1}{\pi}\int_{|z|\leq 1}\!f(z)\mathrm{d}z.
  \]  
  Suppose for now that we have  
  \begin{equation}\label{eq:mom4}  
    \mathbb{E}[\left(S_n-\mathbb{E} S_n\right)^4]=\mathcal{O}\Bigr(\frac{1}{n^2}\Bigr).  
  \end{equation}  
  By monotone convergence or by the Fubini\,--\,Tonelli theorem,  
  \[  
  \mathbb{E}\sum_{n=1}^\infty\left(S_n-\mathbb{E} S_n\right)^4%  
  =\sum_{n=1}^\infty\mathbb{E}[\left(S_n-\mathbb{E} S_n\right)^4]<\infty  
  \]  
  and thus $\sum_{n=1}^\infty\left(S_n-\mathbb{E} S_n\right)^4<\infty$ almost
  surely, which implies $\lim_{n\to\infty}S_n-\mathbb{E} S_n=0$ almost surely.
  Since $\lim_{n\to\infty}\mathbb{E} S_n=S_\infty$ by Theorem
  \ref{th:ginibre-circular-weak}, we get that almost surely
  \[  
  \lim_{n\to\infty}S_n=S_\infty.  
  \]  
  Finally, one can swap the universal quantifiers on $\omega$ and $f$ thanks  
  to the separability of the set of compactly supported continuous bounded  
  functions $\mathbb{C}\to\mathbb{R}$ equipped with the supremum norm. To establish  
  the fourth moment bound \eqref{eq:mom4}, we set  
  \[  
  S_n-\mathbb{E} S_n=\frac{1}{n}\sum_{k=1}^nZ_k %  
  \quad\text{with}\quad %  
  Z_k=
  f\Bigr(\frac{\boldsymbol{\lambda}_k}{\sqrt{n}}\Bigr)
  -\mathbb{E}f\Bigr(\frac{\boldsymbol{\lambda}_k}{\sqrt{n}}\Bigr). 
  \]  
  Next, we obtain, with $\sum_{k_1,\ldots}$ running over distinct indices in  
  $1,\ldots,n$,  
  \begin{align*}  
    \mathbb{E}\left[\left(S_n-\mathbb{E} S_n\right)^4\right] 
    &=\frac{1}{n^4}\sum_{k_1}\mathbb{E}[Z_{k_1}^4] \\  
    &\quad +\frac{4}{n^4}\sum_{k_1,k_2}\mathbb{E}[Z_{k_1}Z_{k_2}^3] \\  
    &\quad +\frac{3}{n^4}\sum_{k_1,k_2}\mathbb{E}[Z_{k_1}^2Z_{k_2}^2] \\  
    &\quad +\frac{6}{n^4}\sum_{k_1,k_2,k_3}\mathbb{E}[Z_{k_1}Z_{k_2}Z_{k_3}^2] \\  
    &\quad +\frac{1}{n^4}\sum_{k_1,k_2,k_3,k_3,k_4}\!\!\!\!\!\!\mathbb{E}[Z_{k_1}Z_{k_3}Z_{k_3}Z_{k_4}].  
  \end{align*}  
  The first three terms of the right are $\mathcal{O}(n^{-2})$ since
  $\max_{1\leq k\leq n}|Z_k|\leq\Vert f\Vert_\infty$. The expressions of
  $\varphi_{n,3}$ and $\varphi_{n,4}$ from Theorem \ref{th:ginibre-k-points}
  allow to show that the remaining two terms are also $\mathcal{O}(n^{-2})$,
  see \cite{MR841088}. %FIXME: check?
\end{proof}

The following theorem includes Theorem \ref{th:ginibre-circular-weak},
which corresponds to the case $k=1$.

\begin{theorem}[Chaoticity]\label{th:chaos}
  Let $\mu_\infty$ be the uniform distribution on the unit disc
  $\{z\in\mathbb{C}:|z|\leq1\}$. For all $1\leq k\leq n$, denoting by
  $P_{n,k}$ the $k$-dimensional marginal distribution of the Ginibre gas
  \eqref{eq:ginibre}, we have
  \[
    P_{n,k}\weak\mu_\infty^{\otimes k}.
  \]
\end{theorem}

\begin{proof}[Idea of the proof]
  The measures $P_{n,k}$ and $\mu_\infty$ have densities $\varphi_{n,k}$ and
  $z\in\mathbb{C}\mapsto\varphi_{\infty}(z)=\pi^{-1}\mathbf{1}_{|z|\leq1}$.
  
  The case $k=1$ is nothing else but Theorem \ref{th:ginibre-circular-weak},
  namely $P_{n,1}\weak\mu_\infty$. This comes via dominated convergence from
  the fact that $\lim_{n\to\infty}\varphi_{n,k}=\varphi_\infty$ uniformly on
  compact subsets of $\{z\in\mathbb{C}:|z|\neq 1\}$.

  Let us consider now the case $k=2$. Here again, by dominated convergence, it
  suffices to show that
  \[
    \lim_{n\to\infty}\varphi^{n,2}=\varphi_{\infty}^{\otimes 2}
  \]
  uniformly on compact subsets of
  $\{(z_1,z_2)\in\mathbb{C}^2:|z_1|\neq1,|z_2|\neq1,z_1\neq z_2\}$.

  By Theorem \ref{th:ginibre-k-points}, for all $z_1,z_2\in\mathbb{C}$,
  \begin{align}
  \varphi^{n,2}(z_1,z_2)
  &=
  \frac{n}{n-1}
  \frac{\mathrm{e}^{-n(|z_1|^2+|z_2|^2)}}{\pi^2}
  \big( \mathrm{e}_n(n|z_1|^2)\mathrm{e}_n(n|z_2|^2)-|\mathrm{e}_n(nz_1\overline{z}_2)|^2 \big)\nonumber\\
  &=\frac{n}{n-1}\varphi^{n,1}(z_1)\varphi^{n,1}(z_2)
  -
  \frac{n}{n-1}
  \frac{\mathrm{e}^{-n(|z_1|^2+|z_2|^2)}}{\pi^2}
  |\mathrm{e}_n(nz_1\overline{z}_2)|^2\label{eq:2points}
  \end{align}
  where $\mathrm{e}_n$ is as in Lemma \ref{le:exp}. It follows that for any
  $n\geq2$ and $z_1,z_2\in\mathbb{C}$,
  \begin{align}\label{defDelta}
  \Delta_n(z_1,z_2)
  &= \varphi^{n,2}(z_1,z_2)-\varphi^{n,1}(z_1)\varphi^{n,1}(z_2)\nonumber\\
  &= \frac{1}{n-1}\varphi^{n,1}(z_1)\varphi^{n,1}(z_2) 
  -\frac{n}{n-1}\frac{\mathrm{e}^{-n(|z_1|^2+|z_2|^2)}}{\pi^2}|\mathrm{e}_n(nz_1\overline{z}_2)|^2.
  \end{align}
  In particular, using $\varphi^{n,2}\geq0$ for the lower bound,
  \[
  -\varphi^{n,1}(z_1)\varphi^{n,1}(z_2)
  \leq
  \Delta_n(z_1,z_2)
  \leq
  \frac{1}{n-1}\varphi^{n,1}(z_1)\varphi^{n,1}(z_2).
  \]
  From this and Lemma \ref{le:exp} we first deduce that for any compact subset
  $K$ of $\{z\in\mathbb{C}:|z|>1\}$
  \[
  \lim_{n\to\infty}
  \sup_{\substack{z_1\in\mathbb{C}\\z_2\in K}}|\Delta_n(z_1,z_2)|
  =
  \lim_{n\to\infty}
  \sup_{\substack{z_1\in K\\z_2\in\mathbb{C}}}|\Delta_n(z_1,z_2)|
  =0.
  \]
  It remains to show that $\Delta_n(z_1,z_2)\to0$ as $n\to\infty$ when $z_1$
  and $z_2$ are in compact subsets of
  $\{(z_1,z_2)\in\mathbb{C}^2:|z_1| < 1, |z_2| < 1\}$. In this case
  $|z_1\overline{z}_2|\leq1$, and Lemma \ref{le:exp} gives
  \[
  |\mathrm{e}_n(nz_1\overline{z}_2)|^2
  \leq 2\mathrm{e}^{2n\Re(z_1\overline{z}_2)}
  +
  2r_n^2(z_1\overline{z}_2).
  \]
  Next, using the elementary identity
  $2\Re(z_1\overline{z}_2)=|z_1|^2+|z_2|^2-|z_1-z_2|^2$, we get
  \begin{equation}\label{enrn}
  \mathrm{e}^{-n(|z_1|^2+|z_2|^2)}|
  \mathrm{e}_n(nz_1\overline{z}_2)|^2
  \leq
  2\mathrm{e}^{-n|z_1-z_2|^2}
  +2\mathrm{e}^{-n(|z_1|^2+|z_2|^2)}r_n^2(z_1\overline{z}_2).
  \end{equation}
  Since $|z_1\overline{z}_2|\leq1$, the formula for $r_n$ in Lemma \ref{le:exp} gives
  \[
  \mathrm{e}^{-n(|z_1|^2+|z_2|^2)}r^2_n(z_1\overline{z}_2) \leq
  \mathrm{e}^{-n(|z_1|^2+|z_2|^2-2-\log|z_1|^2-\log|z_2|^2)} \frac{(n+1)^2}{2\pi n}.
  \]
  Using \eqref{defDelta}, \eqref{enrn} and the bounds
  $\varphi^{n,1} \leq \pi^{-1}$ and $u-1 - \log u >0$ for $0<u<1$, it follows
  that $\Delta_n(z_1, z_2)$ tends to $0$ as $n\to\infty$ uniformly in
  $z_1,z_2$ on compact subsets of
  \[
  \{(z_1,z_2)\in\mathbb{C}^2:|z_1|<1,|z_2|<1,z_1\neq z_2\}.
  \]
  This finishes the proof of the case $k=2$. The case $k\geq3$ follows from
  the case $k=2$ by Lemma \ref{le:chaos}.
\end{proof}

%FIXME: More estimates on the Ginibre kernel are available in Makarov and co. 

\begin{remark}[Impossibility of global uniform convergence of densities]
  The convergence of $\varphi_{n,1}$ cannot hold uniformly on arbitrary
  compact sets of $\mathbb{C}$ since the point-wise limit is not continuous on
  the unit circle. Similarly, the convergence of $\varphi_{n,2}$ cannot hold
  on $\{(z,z):z\in\mathbb{C}, |z|<1\}$ since $\varphi^{n,2}(z,z)=0$ for any
  $n\geq2$ and $z\in\mathbb{C}$ while
  $\varphi_\infty(z)\varphi_\infty(z)=\pi^{-2}\neq0$ when $|z|<1$.
  % This
  % phenomenon is due to the singularity of the interaction.
\end{remark}

\begin{lemma}[Chaoticity]\label{le:chaos}
  Let $E$ be a Polish space. For all $n\geq1$, let $P_n\in\mathcal{M}_1(E^n)$
  be exchangeable, and for all $1\leq k\leq n$, let
  $P_{n,k}\in\mathcal{M}_1(E^k)$ be its $k$-dimensional marginal distribution.
  For all $n\geq1$, let us pick a random vector
  $X_n=(X_{n,1},\ldots,X_{n,n})\sim P_n$ and let us define the random
  empirical measure
  \[
    \mu_n= \displaystyle \frac{1}{n} \sum_{i=1}^n \delta_{X_{n,i}}.
  \]
  For all $\mu\in\mathcal{M}_1(E)$, the following properties are equivalent:
  \begin{enumerate}
  \item $\mu_n\weak\delta_{\mu}$ in $\mathcal{M}_1(\mathcal{M}_1(E))$
    (here we see $\mu_n$ a random variable taking values in
    $\mathcal{M}_1(E)$);
  \item $P_{n,k}\weak\mu^{\otimes k}$ for any fixed $k\geq1$ (note that
    $P_{n,k}$ has a meaning as soon as $n\geq k$);
  \item $P_{n,2}\weak\mu^{\otimes 2}$;
  \end{enumerate}
  where these weak convergences are with respect to continuous and bounded test functions.
\end{lemma}

\begin{proof}
  Folkloric in the domain of mean field particle systems. We refer to
  \cite{MR3847984} and references therein.
\end{proof}

\begin{theorem}[Central limit phenomenon]\label{th:ginibre-clt}
  Let $\mu_n$ be as in Theorem \ref{th:ginibre-circular-weak}. Then, for all
  measurable $f:\mathbb{C}\to\mathbb{R}$ which are $\mathcal{C}^1$ in a
  neighborhood of the unit disc $D=\{z\in\mathbb{C}:|z|\leq1\}$ of the complex
  plane,
  \[
    n\Bigr[\int f\mathrm{d}\mu_n-\mathbb{E}\int f\mathrm{d}\mu_n\Bigr]
    =\sum_{k=1}^n\Bigr[f\Bigr(\frac{\boldsymbol{\lambda}_k}{\sqrt{n}}\Bigr)-\mathbb{E}f\Bigr(\frac{\boldsymbol{\lambda}_k}{\sqrt{n}}\Bigr)\Bigr]
    \overset{\mathrm{d}}{\underset{n\to\infty}{\longrightarrow}}
    \mathcal{N}\Bigr(\frac{1}{4\pi}\|f\|_{\mathrm{H}^1(D)}^2
    +\frac{1}{2}\|f\|_{\mathrm{H}^{1/2}(\partial D)}^2
    \Bigr)
  \]
  where 
  \[
    \|f\|_{\mathrm{H}^1(D)}^2
    =\int_{D}|\nabla f|^2\mathrm{d}z
    \quad\text{and}\quad
    \|f\|_{\mathrm{H}^{1/2}(\partial D)}^2
    =\sum_{k\in\mathbb{Z}}|k||\widehat{f}(k)|^2
  \]
  where $\widehat{f}(k)$ is the $k$-th Fourier coefficient of $f$ on
  $\partial D=\{z\in\mathbb{C}:|z|=1\}$, namely
  \[
    \widehat{f}(k)=\frac{1}{2\pi}\int_0^{2\pi}f(\mathrm{e}^{\mathrm{i}\theta})\mathrm{e}^{-\mathrm{i}k\theta}\mathrm{d}\theta.
  \]
\end{theorem}

Note that $\|f\|_{\mathrm{H}^{1/2}(\partial D)}=0$ if $f$ is analytic on a
neighborhood of $\partial D$.

% In particular, with $f=\log|\cdot-z|$,
%FIXME: comments on characteristic polynomial and GFF.

\begin{proof}[About the proof]
  It is known \cite{MR3155254} that the cumulants of linear statistics of
  determinantal processes have a nice form that can be used to prove a central
  limit theorem. This idea is followed in \cite{MR2361453} in order to produce
  the result, via combinatorial identities, and via reduction to polynomial
  test functions. The method is used for more general two-dimensional
  determinantal gases in \cite{MR2817648,MR3342661}.
\end{proof}

\begin{theorem}[Distribution of the moduli in the Ginibre
  model]\label{th:kostlan}
  Let $(\boldsymbol{\lambda}_1,\ldots,\boldsymbol{\lambda}_n)$ be the
  exchangeable random vector considered in Theorem \ref{th:ginibre}. Then the
  following equality in distribution holds
  \[
    (|\boldsymbol{\lambda}_1|,\ldots,|\boldsymbol{\lambda}_n|)
    \overset{\mathrm{d}}{=}
    (Z_{\sigma(1)},\ldots,Z_{\sigma(n)})
  \]
  where $Z_1,\ldots,Z_n$ are independent non-negative random variables
  with\footnote{The law $\mathrm{Gamma}(a,\lambda)$ has density
	$x\in\mathbb{R}\mapsto \frac{\lambda^a}{\Gamma(a)}x^{a-1}\mathrm{e}^{-\lambda
      x}\mathbf{1}_{x\geq0}$, and
    $\mathrm{Gamma}(a,\lambda)*\mathrm{Gamma}(b,\lambda)=\mathrm{Gamma}(a+b,\lambda)$.}\footnote{Note
    that
    $(\sqrt{2}Z_k)^2\sim\mathrm{Gamma}(k,\frac{1}{2})=\mathrm{Exponential}(\frac{1}{2})^{*k}=\chi^2(2k)$
    since $\chi^2(n)=\mathrm{Gamma}(\frac{n}{2},\frac{1}{2})$ for all
    $n\geq1$.}\footnote{For $n=1$ we recover the Box\,--\,Muller formula
    $|X|^2\sim\chi^2(2)=\mathrm{Gamma}(1,\frac{1}{2})=\mathrm{Exponential}(\frac{1}{2})$
    with $X\sim\mathcal{N}(0,I_2)$.}
  \[
    Z_k^2\sim\mathrm{Gamma}(k,1),\quad 1\leq k\leq n,
  \]
  and where $\sigma$ is a uniform random permutation of $\{1,\ldots,n\}$
  independent of $Z_1,\ldots,Z_n$. Equivalently, for all symmetric bounded
  measurable $F:\mathbb{R}^n\to\mathbb{R}$, we have
  \( \mathbb{E}(F(|\boldsymbol{\lambda}_1|,\ldots,|\boldsymbol{\lambda}_n|))
  =\mathbb{E}(F(Z_1,\ldots,Z_n)). \)
\end{theorem}

This is an equality between two exchangeable laws on $\mathbb{R}^n$, in other
words an equality in law between two configurations of unlabeled
random points in $\mathbb{R}$ (multi-sets). Note in particular that for all
$1\leq k\leq n$, taking
$F(x_1,\ldots,x_n)=\sum_{\substack{i_1,\ldots,i_k\\\text{
      distinct}}}f(x_{i_1})\cdots f(x_{i_k})$ gives equality of $k$-point
correlation functions.

% It is not enough to take a look at empirical measures (k=1)
% We need empirical measures of couples, and more generally of all k-uples.

\begin{proof}% [Idea of proof]
  From Theorem \ref{th:ginibre}, the exchangeable random vector
  $(\boldsymbol{\lambda}_1,\ldots,\boldsymbol{\lambda}_n)$ has density
  $\varphi_n$. It follows that the density of the exchangeable random vector
  $(|\boldsymbol{\lambda}_1|,\ldots,|\boldsymbol{\lambda}_n|)$ is obtained from
  $\varphi_n$ by integrating the phases in polar coordinates. In polar
  coordinates $x_k=r_k\mathrm{e}^{\mathrm{i}\theta_k}$, the density
  $\varphi_n$ writes
  \[
    (r_1,\ldots,r_n,\theta_1,\ldots,\theta_n)
    \mapsto
    \frac{\mathrm{e}^{-\sum_{j=1}^nr_j^2}}{\pi^n\prod_{k=1}^nk!}
    \prod_{j<k}|r_j\mathrm{e}^{\mathrm{i}\theta_j}-r_k\mathrm{e}^{\mathrm{i}\theta_k}|^2.   
  \]
  Now we have, denoting $\Sigma_n$ the symmetric group of permutations of $\{1,\ldots,n\}$,
  \begin{align*}
    \prod_{j<k}|r_j\mathrm{e}^{\mathrm{i}\theta_j}-r_k\mathrm{e}^{\mathrm{i}\theta_k}|^2
    &=
    \prod_{j<k}(r_j\mathrm{e}^{\mathrm{i}\theta_j}-r_k\mathrm{e}^{\mathrm{i}\theta_k})
    \prod_{j<k}\overline{(r_j\mathrm{e}^{\mathrm{i}\theta_j}-r_k\mathrm{e}^{\mathrm{i}\theta_k})}\\
    &=
    \det\Bigr[r_j^{k-1}\mathrm{e}^{\mathrm{i}(k-1)\theta_j}\Bigr]_{1\leq j,k\leq n}
    \det\Bigr[r_j^{k-1}\mathrm{e}^{-\mathrm{i}(k-1)\theta_j}\Bigr]_{1\leq
      j,k\leq n}\\
    &=\Bigr(\sum_{\sigma\in\Sigma_n}(-1)^{\mathrm{sign}(\sigma)}\prod_{j=1}^nr_j^{\sigma(j)-1}\mathrm{e}^{\mathrm{i}(\sigma(j)-1)\theta_j}\Bigr)
      \Bigr(\sum_{\sigma'\in\Sigma_n}(-1)^{\mathrm{sign}(\sigma')}\prod_{j=1}^nr_{j}^{\sigma'(j)-1}\mathrm{e}^{-\mathrm{i}(\sigma'(j)-1)\theta_{j}}\Bigr)\\
    &=\sum_{\sigma,\sigma'\in\Sigma_n}(-1)^{\mathrm{sign}(\sigma)+\mathrm{sign}(\sigma')}\prod_{j=1}^nr_j^{\sigma(j)+\sigma'(j)-2}\mathrm{e}^{\mathrm{i}((\sigma(j)-\sigma'(j))\theta_{j})}.
  \end{align*}
  If we integrate the phases, we note that only the terms with
  $\sigma=\sigma'$ contribute to the result, namely
  \[
    \int_0^{2\pi}\cdots\int_0^{2\pi}
    \prod_{j<k}|r_j\mathrm{e}^{\mathrm{i}\theta_j}-r_k\mathrm{e}^{\mathrm{i}\theta_k}|^2
    \mathrm{d}\theta_1\cdots\mathrm{d}\theta_n
    =(2\pi)^n\sum_{\sigma\in\Sigma_n}\prod_{j=1}^nr_j^{2(\sigma(j)-1)}
    =(2\pi)^n\mathrm{per}\Bigr[r_j^{2k}\Bigr]_{1\leq j,k\leq n}
  \]
  where ``per'' stands for ``permanent''. Therefore, the (exchangeable) density
  of the moduli is given by
  \[
    \int_0^{2\pi}\cdots\int_0^{2\pi}
    \frac{\mathrm{e}^{-\sum_{j=1}^nr_j^2}}{\pi^n\prod_{k=1}^nk!}
    \prod_{j<k}|r_j\mathrm{e}^{\mathrm{i}\theta_j}-r_k\mathrm{e}^{\mathrm{i}\theta_k}|^2
    \mathrm{d}\theta_1\cdots\mathrm{d}\theta_n
    =
    \mathrm{perm}\Bigr[\frac{2}{k!}r_j^{2k}\mathrm{e}^{-r_j^2}\Bigr]_{1\leq
      j,k\leq n}.
  \]
  But if $f_1,\ldots,f_n:\mathbb{R}\to\mathbb{R}$ are probability density
  functions then
  $(x_1,\ldots,x_n)\mapsto\mathrm{perm}[f_j(x_k)]_{1\leq j,k\leq n}$ is the
  density of the random vector $(X_{\sigma(1)},\ldots,X_{\sigma(n)})$ where
  $X_1,\ldots,X_n$ are independent real random variables with densities
  $f_1,\ldots,f_n$ and where $\sigma$ is a uniform random permutation of
  $\{1,\ldots,n\}$ independent of $X_1,\ldots, X_n$. Also the desired result
  follows from the formula above and the fact that for all $1\leq k\leq n$, a
  non-negative random variable $Z_k$ has density
  $r\mapsto\frac{2}{k!}r^{2k}\mathrm{e}^{-r^2}\mathbf{1}_{r\geq0}$ if and only
  if $Z_k^2\sim\mathrm{Gamma}(k,1)$.
  
  This proof, essentially due to Kostlan \cite{MR1148410}, see also
  \cite{MR1986426}, relies on the determinantal nature of $\varphi_n$ in
  \eqref{eq:ginibre}, and remains usable for general determinantal processes,
  see for instance \cite{MR2552864}.
\end{proof}

%Convergence and fluctuation of the spectral radius.

\begin{theorem}[Spectral radius]
  \label{th:ginibre-spectral-radius}
  With the notation of Theorem \ref{th:ginibre}, almost surely 
  \[
    \rho_n=\max_{1\leq k\leq n}\frac{|\boldsymbol{\lambda}_k|}{\sqrt{n}}
    \underset{n\to\infty}{\longrightarrow}1.
  \]
  Moreover, denoting $\kappa_n=\log\frac{n}{2\pi}-2\log(\log(n))$, 
  \[
    \sqrt{4n\kappa_n} \Bigr(\rho_n-1-\sqrt{\frac{\kappa_n}{4n}}\Bigr)
    \underset{n\to\infty}{\overset{\mathrm{d}}{\longrightarrow}}
    \mathrm{Gumbel}\footnote{If $X\sim\mathrm{Exp}(1)$ then
      $-\log(X)$ has Gumbel law and cumulative distribution function
      $\mathbb{P}(-\log(X)\leq x)=\mathrm{e}^{-\mathrm{e}^{-x}}$.}.
  \]
\end{theorem}

Note the the second statement (Gumbel fluctuation) implies that
$\rho_n\underset{n\to\infty}{\longrightarrow}1$ in probability.

Regarding the convergence, see \cite{bordenave-chafai-garcia-zelada} for a
random analytic function point of view, related to the central limit theorem.
Regarding the fluctuation, see \cite{MR2288065,MR2594353} for an interpolation
with the Tracy\,--\,Widom fluctuation at the edge of GUE.

\begin{proof}[Idea of proof] By Theorem \ref{th:kostlan} Since
  $Z_k^2\overset{\mathrm{d}}{=}E_1+\cdots+E_k$ where $E_1,\ldots,E_k$ are
  independent and identically distributed exponential random variables of unit
  mean, we get, for every $r>0$,
  \[  
    \mathbb{P}(\rho_n\leq \sqrt{n}r) %  
    =\prod_{1\leq k\leq n}\mathbb{P}\left(\frac{E_1+\cdots+E_k}{n}\leq r^2\right).
  \]  
  By the law of large numbers, this tends as $n\to\infty$ to $0$ or $1$
  depending on the position of $r$ with respect to $1$. Moreover the central
  limit theorem suggests that $\rho_n$ behaves as $n\to\infty$ as the maximum
  of independent and identically distributed Gaussian random variables, a
  situation for which it is known that the fluctuation follows the Gumbel law.
  The full proof is in \cite{MR1986426}
  % \footnote{Apparently the author, Brian
  % Rider, was not aware of the work of Eric Kostlan \cite{MR1148410} so he
  % simply reinvented the argument.}
  and involves crucially a quantitative central limit theorem and the
  Borel\,--\,Cantelli lemma. The approach is robust and remains valid beyond the
  Ginibre gas, for determinantal gases, see for instance
  \cite{MR3215627,MR3615091,david-edge} and references therein.
  % Let us mention that
  % the convergence of the spectral radius $\rho_n$ was already obtained in
  % \cite[ch.\ 15]{MR2129906} by integrating the joint density
  % $\varphi_n(x_1,\ldots,x_n)$ over $\bigcap_{1\leq i\leq n}\{|x_i|>r\}$.
  % , and the same argument is reproduced in Hwang \cite{MR841088}.
\end{proof}

\begin{remark}[Real or quatertionic Ginibre model]
  How about an analogue of Theorem \ref{th:ginibre} when the entries of
  $\mathbf{M}$ are real Gaussian or real quaternionic Gaussian instead of
  complex Gaussian? Some answers are already in \cite{MR0173726}. In these
  cases, the density of the eigenvalues can be computed but it is not the beta
  gas \eqref{eq:beta} with $\beta\in\{1,4\}$. This is in contrast with the
  G(O|U|S)E triplet of the Hermtian random matrix Dysonian universe
  \cite{MR0177643}. See for instance \cite{MR1437734,dubach-quaternions} and
  references therein.
\end{remark}

\begin{remark}[Large deviations] The large deviations principle for the beta
  Ginibre gas \eqref{eq:ginibre} was established in
  \cite{MR1746976,MR1606719}, using a method inspired from \cite{BG08}, itself
  inspired from \cite{MR1228526,MR1296352}. It does not rely on the
  determinantal structure, and allows to extend Theorem
  \ref{th:ginibre-circular-strong} to all $\beta>0$.
\end{remark}

\subsection{More determinantal models}\label{se:more}

It is well known that the ratio of two independent real standard Gaussian
random variables follows a Cauchy distribution. The following theorem can be
seen as a matrix version of this phenomenon.

\begin{theorem}[Forrester\,--\,Krishnapur spherical ensemble]\label{th:spherical}
  Let $\mathbf{M}_1$ and $\mathbf{M}_2$ be independent copies of the Ginibre
  random matrix defined in \eqref{eq:ginibre-matrix}. Then as an exchangeable
  random vector of $\mathbb{C}^n$, the eigenvalues of
  $\mathbf{M}_1{\mathbf{M}_2}^{-1}$ have density
  \[
    (z_1,\ldots,z_n)\in\mathbb{C}^n
    \mapsto\frac{1}{Z_n}\frac{\prod_{j<k}|z_j-z_k|^2}{\prod_{j=1}^n(1+|z_j|^2)^{n+1}}.
    %=\frac{\mathrm{e}^{-(n+1)\sum_{j=1}^n\log(1+|z_j|^2)}}{Z_n}\prod_{j<k}|z_j-z_k|^2.
  \]
  This corresponds to the beta gas \eqref{eq:beta} with
  $V=\frac{1}{2}\frac{n+1}{n}\log(1+\left|\cdot\right|^2)$ and $\beta=2$.
  Moreover its push-forward on the Riemann sphere using inverse stereographic
  projection is the uniform law on the sphere. % see \cite{MR2926763}.
\end{theorem}

See \cite{MR2489167} and \cite{MR2552864,forrester-krishnapur,MR2641363} for a
proof. The set of singular $n\times n$ complex matrices is a hyper-surface of
zero Lebesgue measure in $\mathbb{C}^{n^2}$ and therefore, almost surely, the
Ginibre random matrix $\mathbf{M}$ in \eqref{eq:ginibre-matrix} is invertible
(its law is absolutely continuous with respect to the Lebesgue measure on
$\mathbb{C}^{n^2}$).

From Theorem \ref{th:mustar}, the equilibrium measure of the gas is heavy
tailed with density
\[
  z\in\mathbb{C}\mapsto\frac{1}{\pi(1+|z|^2)^2}.
\]
A large deviations principle for the empirical measure associated to the
Coulomb gas is proved in \cite{MR2926763} in relation with the sphere. The
convergence of the empirical measure is also considered in \cite{ECP2011-10}.
The fluctuation at the edge is discussed in \cite{MR3215627} and studied in
\cite{MR3615091} by using the idea of Kostlan behind Theorem \ref{th:kostlan}
thanks to the determinantal structure. A beta version of the model is
considered in \cite{MR3831027} and studied using transportation of measure.

\begin{theorem}[\.{Z}yczkowski\,--\,Sommers ensemble]\label{th:unitary}
  Let $\mathbf{U}={(\mathbf{U}_{j,k})}_{1\leq j,k\leq m}$ be a random
  $m\times m$ unitary matrix following the (Haar) uniform law on this compact
  group of matrices. Then, for all $1\leq n<m$, as an exchangeable random
  vector of $\mathbb{C}^n$, the eigenvalues of the truncation
  ${(\mathbf{U}_{j,k})}_{1\leq j,k\leq n}$ have density
  \[
    (z_1,\ldots,z_n)\in\mathbb{C}^n
    \mapsto
    \frac{\prod_{j=1}^n(1-|z_j|^2)^{m-n-1}}{Z_n}\prod_{1\leq j<k\leq n}|z_j-z_k|^2.
  \]
  This corresponds to the beta gas \eqref{eq:beta} with
  $V=V_{n,m}=\frac{m-n-1}{n}\log\frac{1}{1-|z|^2}$ and $\beta=2$. 
\end{theorem}

See \cite{MR1748745} for a proof, and \cite{forrester-krishnapur} for the
special case $m\geq 2n$ and a link with the pseudo-sphere and Schur
transformation. Following \cite{MR2198697} and references therein, if
$\lim_{m,n\to\infty}\frac{n}{m}=\alpha\in(0,1)$ then the empirical measure
converges towards the heavy tailed probability measure with density
\[
  z\in\mathbb{C}\mapsto\frac{(1-\alpha)}{\pi\alpha(1-|z|^2)^2}
  \mathbf{1}_{|z|\leq\sqrt{\alpha}}.
\]
In a sense this law interpolates between the uniform law on the unit disc
($\alpha\to0$ after scaling by $\sqrt{\alpha}$) and the uniform law on the
unit circle ($\alpha\to1$). A large deviations principle is obtained in
\cite{MR2198697}, concentration inequalities are derived in
\cite{zbMATH07107321}, while the fluctuation at the edge is studied in
\cite{MR3615091}.

\begin{theorem}[Product of Ginibre random matrices]\label{th:product}
  Let $m\geq1$ and let $\mathbf{M}_1,\ldots,\mathbf{M}_m$ be independent and
  identically distributed copies of the $n\times n$ in
  \eqref{eq:ginibre-matrix}. Then, as an exchangeable random vector of
  $\mathbb{C}^n$, the eigenvalues of the scaled product
  $n^{-\frac{m}{2}}\mathbf{M}_1\cdots\mathbf{M}_m$ have density
  \[
    \frac{\prod_{j=1}^nw_m(\sqrt{n}|z_j|)}{Z_n}\prod_{j<k}|z_j-z_k|^2
    %=\frac{\mathrm{e}^{-\sum_{j=1}^n\log w_m(|z_j|)}}{Z_n}\prod_{j<k}|z_j-z_k|^2.
  \]
  where $w_k$ is the Meijer G-function given by the recursive formula
  \[
    w_1(z)=\mathrm{e}^{-|z|^2}
    \quad\text{and}\quad
    w_k(z)=2\pi\int_0^\infty w_{k-1}
    \Bigr(\frac{z}{r}\Bigr)\frac{\mathrm{e}^{-r^2}}{r}\mathrm{d}r.
  \]
  This corresponds to the beta gas \eqref{eq:beta} with
  $V=V_{n,m}=-\frac{1}{n}\log w_m(\sqrt{n}\bullet)$ and $\beta=2$.
\end{theorem}

See \cite{MR2993423} for a proof. Following
\cite{gotze-tikhomirov,ECP2011-10}, its converges to the equilibrium measure
with density
\[
  z\in\mathbb{C}\mapsto\frac{|z|^{\frac{2}{m}-2}}{m\pi}\mathbf{1}_{|z|\leq1}.
\]
We recover the uniform law on the unit disc when $m=1$. The edge fluctuation
is considered in \cite{MR3615091}.

\begin{remark}[Determinantal gases and random normal matrices]
  An $n\times n$ complex matrix is \emph{normal} when $MM^*=M^*M$. The random
  matrices in theorems
  \ref{th:ginibre},\ref{th:spherical},\ref{th:unitary},\ref{th:product} are
  not normal. Let us comment now on models of normal random matrices. Let
  $\mathcal{N}_n$ be the hyper-surface of $\mathbb{C}^{n^2}$ of all
  $n\times n$ normal matrices. Let $V:\mathbb{C}\to\mathbb{R}$ be
  $\mathcal{C}^2$ and such that $V(z)\geq c\log(1+|z|^2)$ for some constant
  $c>0$. Following \cite{MR1643533,MR2172690}, let us consider the probability
  measure on $\mathcal{N}_n$ with density proportional to
  $M\mapsto\mathrm{e}^{-n\mathrm{Trace}(V(M))}$ with respect to the Hausdorff
  measure on $\mathcal{N}_n$. This produces random normal matrices, and their
  eigenvalues, seen as an exchangeable random vector, have density given by
  the gas \eqref{eq:beta} with $\beta=2$. This random (normal) matrix model is
  referred to as the \emph{random normal matrix model}. The fluctuation of the
  empirical measure is studied in \cite{MR3342661,MR2817648}, while the
  fluctuation at the edge is studied in \cite{MR3215627,MR3615091,david-edge}.
\end{remark}

The power of a Ginibre matrix has also a nice determinantal structure, see \cite{MR3878136}.

\section{Comments and open problems}

We have skipped several important old and new results on Coulomb gases. The
main themes are local versus global, first versus second order, macroscopics
versus microscopics, non-universal versus universal.

\textbf{Universality.} The first order global convergence
$\lim_{n\to\infty}\mu_n=\mu_V$, that we call macroscopics, is not universal in
the sense that the limit $\mu_V$ still depends on $V$. The second order
convergence provided by the central limit theorem \eqref{eq:CLT} is universal
in the sense that the limit should not depend on $V$. Similarly, for a
two-dimensional Coulomb gas with radial confining potential $V$, the limit of
the edge depends on $V$ but its fluctuation does not and is universal.
Universality emerges often in a second order asymptotic analysis, as for
classical limit theorems of probability theory.

\textbf{Microscopics.} A second order analysis corresponds to the asymptotic
analysis of $n(\mu_n-\nu_V)$ as $n\to\infty$. This can be seen as a
microscopic analysis while the convergence $\mu_n\to\mu_V$ is a macroscopic
analysis. This corresponds to a second order Taylor formula for the quadratic
form $\mathcal{E}_V$, in other words in a special factorization, leading to a
new object called the \emph{renormalized energy}. This was the subject of an
series of works by Étienne Sandier and Sylvia Serfaty, and by Sylvia Serfaty
and other co-authors. See for instance
\cite{MR3888701,MR3837028,MR3735628,MR3309890} and references therein. An
outcome of this refined analysis is a second order asymptotics for the free
energy. More precisely, recall that the Boltzmann\,--\,Shannon entropy of the
Boltzmann\,--\,Gibbs measure $P_n$ in \eqref{eq:P} is defined by
\[
  \mathcal{S}(P_n)
  =-\int_{(\mathrm{R}^d)^n}f_n(x_1,\cdots,x_n)\log f_n(x_1,\ldots,x_n)
  \mathrm{d}x_1\cdots\mathrm{d}x_n
\]
where $f_n$ is the density of $P_n$. Its Helmholtz free energy is given by
\[
  \int E_n\mathrm{d}P_n
  -\frac{\mathcal{S}(P_n)}{\beta}
  =-\frac{\log Z_n}{\beta},
\]
see \cite{MR3434251}. Now following \cite{MR3735628}, if $\mu_V$ has density
$f_V$ with a finite Boltzmann\,--\,Shannon entropy
\[
  \mathcal{S}(\mu_V)=-\int_{\mathbb{R}^d} f_V\log f_V\mathrm{d}x,
\]
then we have an asymptotic expansion of the free energy as $n\to\infty$ as
\begin{equation*}\label{eq:Zexp}
  -\frac{\log Z_n}{\beta}
  =
  \begin{cases}
    \displaystyle
    \frac{n^2}{2}\mathcal{E}_V(\mu_V)
    -\frac{n\log n}{4}
    +n(c_\beta+c_\beta'\mathcal{S}(\mu_V))
    +no_n(1)
    &
    \text{if $d=2$}\\[1em]
    \displaystyle
    \frac{n^2}{2}\mathcal{E}_V(\mu_V)
    +n(c_{\beta,d,V}+c'_{\beta}\mathcal{S}(\mu_V))
    +no_n(1)
    &\text{if $d\neq2$}
  \end{cases},
\end{equation*}
where $c_\beta,c_\beta',c_{\beta,d,V}$ are constants which can be made
explicit.

\textbf{Edge.} The most elementary open question related to Coulomb gases is
perhaps the law of fluctuation at the edge, even in the case of rotationally
invariant confining potential for arbitrary values of $d$ and $\beta$. The
Gumbel fluctuation is known for instance to be universal for a class of two
dimensional ($d=2$) determinantal ($\beta=2$) Coulomb gases with radial
confining potential, see \cite{MR3215627}. The same question for arbitrary
$\beta$ is open, and the same question for arbitrary dimension $d\geq3$ and
$\beta>0$ is also open.

\textbf{Crystallization.} A conjecture related to Coulomb gases is the
emergence of rigid structures at low temperatures. This is known as
crystallization and was proved in special cases, for instance for
one-dimensional Coulomb gases. See for instance
\cite{MR3309890,MR3429164,MR3888701,MR3837028,zbMATH07204762} and references
therein.

\textbf{More.} Among all the important results on Coulomb gases that we have
not yet mentioned, we may cite the approximate transport maps for universality
considered in \cite{MR3646880,MR3351052}, the rigidity analysis for
hierarchical Coulomb gases considered in \cite{MR4026615}, the local density
for two-dimensional Coulomb gases considered in \cite{MR3694026}, the
Dobrushin\,--\,Lanford\,--\,Ruelle equations considered in \cite{MR4178183},
the Coulomb gas properties on the sphere considered in \cite{MR3902480}, the
local laws and rigidity considered in \cite{zbMATH07310787}, the
quasi-Monte-Carlo method on the sphere considered in \cite{berman-spherical},
and the Berezinskii\,--\,Kosterlitz\,--\,Thouless transition
\cite{kharash-peled,garban-sepulveda}.
% Charles and Estienne - Entanglement entropy and Berezin-Toeplitz operators
% \cite{charles-estienne}%
% Rougerie the Laughlin function and \cite{rougerie-et-al}%

\section*{Acknowledgments}

We would like to thank David García-Zelada and Kilian Rashel for their helpful
remarks.

\bibliographystyle{smfalpha}
\bibliography{erms-2018-chafai}

\providecommand{\bysame}{\leavevmode ---\ }
\providecommand{\og}{``}
\providecommand{\fg}{''}
\providecommand{\smfandname}{\&}
\providecommand{\smfedsname}{\'eds.}
\providecommand{\smfedname}{\'ed.}
\providecommand{\smfmastersthesisname}{M\'emoire}
\providecommand{\smfphdthesisname}{Th\`ese}
\begin{thebibliography}{LACTMS19}

\bibitem[AB12]{MR2993423}
{\scshape G.~Akemann {\normalfont \smfandname} Z.~Burda} -- {\og Universal
  microscopic correlation functions for products of independent {G}inibre
  matrices\fg}, \emph{J. Phys. A} \textbf{45} (2012), no.~46, p.~465201, 18.

\bibitem[AB19]{zbMATH07088025}
{\scshape G.~{Akemann} {\normalfont \smfandname} S.-S. {Byun}} -- {\og {The
  high temperature crossover for general 2D Coulomb gases}\fg}, \emph{{J. Stat.
  Phys.}} \textbf{175} (2019), no.~6, p.~1043--1065 (English).

\bibitem[ABDF11]{MR2920518}
{\scshape G.~Akemann, J.~Baik {\normalfont \smfandname} P.~Di~Francesco}
  (\smfedsname) -- \emph{The {O}xford handbook of random matrix theory}, Oxford
  University Press, Oxford, 2011.

\bibitem[AGZ10]{MR2760897}
{\scshape G.~W. Anderson, A.~Guionnet {\normalfont \smfandname} O.~Zeitouni} --
  \emph{An introduction to random matrices}, Cambridge Studies in Advanced
  Mathematics, vol. 118, Cambridge University Press, Cambridge, 2010.

\bibitem[AHM11]{MR2817648}
{\scshape Y.~Ameur, H.~k. Hedenmalm {\normalfont \smfandname} N.~Makarov} --
  {\og Fluctuations of eigenvalues of random normal matrices\fg}, \emph{Duke
  Math. J.} \textbf{159} (2011), no.~1, p.~31--81.

\bibitem[AHM15]{MR3342661}
{\scshape Y.~Ameur, H.~Hedenmalm {\normalfont \smfandname} N.~Makarov} -- {\og
  Random normal matrices and {W}ard identities\fg}, \emph{Ann. Probab.}
  \textbf{43} (2015), no.~3, p.~1157--1201.

\bibitem[Ame21]{10.1214/21-EJP613}
{\scshape Y.~Ameur} -- {\og {A localization theorem for the planar Coulomb gas
  in an external field}\fg}, \emph{Electronic Journal of Probability}
  \textbf{26} (2021), no.~none, p.~1 -- 21.

\bibitem[AS19]{armstrong-serfaty}
{\scshape S.~Armstrong {\normalfont \smfandname} S.~Serfaty} -- {\og Thermal
  approximation of the equilibrium measure and obstacle problem\fg}, preprint
  \href{https://arxiv.org/abs/1912.13018}{arXiv:1912.13018}, 2019.

\bibitem[AS21]{zbMATH07310787}
{\scshape S.~{Armstrong} {\normalfont \smfandname} S.~{Serfaty}} -- {\og {Local
  laws and rigidity for Coulomb gases at any temperature}\fg}, \emph{{Ann.
  Probab.}} \textbf{49} (2021), no.~1, p.~46--121 (English).

\bibitem[BAG97]{MR1465640}
{\scshape G.~Ben~Arous {\normalfont \smfandname} A.~Guionnet} -- {\og Large
  deviations for {W}igner's law and {V}oiculescu's non-commutative entropy\fg},
  \emph{Probab. Theory Related Fields} \textbf{108} (1997), no.~4, p.~517--542.

\bibitem[BAG08]{BG08}
{\scshape G.~Ben~Arous {\normalfont \smfandname} A.~Guionnet} -- {\og The
  spectrum of heavy tailed random matrices\fg}, \emph{Comm. Math. Phys.}
  \textbf{278} (2008), no.~3, p.~715--751.

\bibitem[BBL96]{MR1418808}
{\scshape E.~Bogomolny, O.~Bohigas {\normalfont \smfandname} P.~Leboeuf} --
  {\og Quantum chaotic dynamics and random polynomials\fg}, \emph{J. Statist.
  Phys.} \textbf{85} (1996), no.~5-6, p.~639--679.

\bibitem[BBLu92]{MR1160289}
{\scshape E.~Bogomolny, O.~Bohigas {\normalfont \smfandname} P.~Leb\oe~uf} --
  {\og Distribution of roots of random polynomials\fg}, \emph{Phys. Rev. Lett.}
  \textbf{68} (1992), no.~18, p.~2726--2729.

\bibitem[BBNY17]{MR3694026}
{\scshape R.~Bauerschmidt, P.~Bourgade, M.~Nikula {\normalfont \smfandname}
  H.-T. Yau} -- {\og Local density for two-dimensional one-component
  plasma\fg}, \emph{Comm. Math. Phys.} \textbf{356} (2017), no.~1, p.~189--230.

\bibitem[BBNY19]{MR4063572}
\bysame , {\og The two-dimensional {C}oulomb plasma: quasi-free approximation
  and central limit theorem\fg}, \emph{Adv. Theor. Math. Phys.} \textbf{23}
  (2019), no.~4, p.~841--1002.

\bibitem[BCF18]{MR3847984}
{\scshape F.~Bolley, D.~Chafa\"{i} {\normalfont \smfandname} J.~Fontbona} --
  {\og Dynamics of a planar {C}oulomb gas\fg}, \emph{Ann. Appl. Probab.}
  \textbf{28} (2018), no.~5, p.~3152--3183.

\bibitem[BCGZ20]{bordenave-chafai-garcia-zelada}
{\scshape C.~Bordenave, D.~Chafaï {\normalfont \smfandname} D.~García-Zelada}
  -- {\og Convergence of the spectral radius of a random matrix through its
  characteristic polynomial\fg}, preprint
  \href{https://arXiv.org/abs/2012.05602}{arXiv:2012.05602} to appear in
  {P}robability {T}heory and {R}elated {F}ields, 2020.

\bibitem[BD20]{zbMATH07202715}
{\scshape P.~{Bourgade} {\normalfont \smfandname} G.~{Dubach}} -- {\og {The
  distribution of overlaps between eigenvectors of Ginibre matrices}\fg},
  \emph{{Probab. Theory Relat. Fields}} \textbf{177} (2020), no.~1-2,
  p.~397--464 (English).

\bibitem[Ben10]{MR2594353}
{\scshape M.~Bender} -- {\og Edge scaling limits for a family of
  non-{H}ermitian random matrix ensembles\fg}, \emph{Probab. Theory Related
  Fields} \textbf{147} (2010), no.~1-2, p.~241--271.

\bibitem[Ber18a]{berman}
{\scshape R.~Berman} -- {\og The {C}oulomb gas, potential theory and phase
  transitions\fg}, preprint
  \href{https://arxiv.org/abs/1811.10249}{arXiv:1811.10249}, 2018.

\bibitem[Ber18b]{MR3825945}
{\scshape R.~J. Berman} -- {\og On large deviations for {G}ibbs measures, mean
  energy and gamma-convergence\fg}, \emph{Constr. Approx.} \textbf{48} (2018),
  no.~1, p.~3--30.

\bibitem[Ber19a]{berman-conc}
{\scshape R.~Berman} -- {\og Sharp deviation inequalities for the 2{D}
  {C}oulomb gas and {Q}uantum hall states, {I}\fg}, preprint
  \href{http://arxiv.org/abs/1906.08529v1}{arXiv:1906.08529v1}, 2019.

\bibitem[Ber19b]{berman-spherical}
{\scshape R.~J. Berman} -- {\og The spherical ensemble and
  quasi-{M}onte-{C}arlo designs\fg}, preprint
  \href{http://arxiv.org/abs/1906.08533v1}{arXiv:1906.08533v1}, 2019.

\bibitem[BFG15]{MR3351052}
{\scshape F.~Bekerman, A.~Figalli {\normalfont \smfandname} A.~Guionnet} --
  {\og Transport maps for {$\beta$}-matrix models and universality\fg},
  \emph{Comm. Math. Phys.} \textbf{338} (2015), no.~2, p.~589--619.

\bibitem[BG99]{MR1678526}
{\scshape T.~Bodineau {\normalfont \smfandname} A.~Guionnet} -- {\og About the
  stationary states of vortex systems\fg}, \emph{Ann. Inst. H. Poincar\'{e}
  Probab. Statist.} \textbf{35} (1999), no.~2, p.~205--237.

\bibitem[BGG17]{MR3668648}
{\scshape A.~Borodin, V.~Gorin {\normalfont \smfandname} A.~Guionnet} -- {\og
  Gaussian asymptotics of discrete {$\beta$}-ensembles\fg}, \emph{Publ. Math.
  Inst. Hautes \'{E}tudes Sci.} \textbf{125} (2017), p.~1--78.

\bibitem[BGL14]{MR3155209}
{\scshape D.~Bakry, I.~Gentil {\normalfont \smfandname} M.~Ledoux} --
  \emph{Analysis and geometry of {M}arkov diffusion operators}, Grundlehren der
  Mathematischen Wissenschaften [Fundamental Principles of Mathematical
  Sciences], vol. 348, Springer, Cham, 2014.

\bibitem[BGZ18]{raphael-david}
{\scshape R.~Butez {\normalfont \smfandname} D.~Garcia-Zelada} -- {\og Extremal
  particles of two-dimensional {C}oulomb gases and random polynomials on a
  positive background\fg}, preprint
  \href{http://arxiv.org/abs/1811.12225}{arXiv:1811.12225} to appear in
  {A}nnals of {A}pplied {P}robability, 2018.

\bibitem[BGZNW21]{butez2021universality}
{\scshape R.~Butez, D.~García-Zelada, A.~Nishry {\normalfont \smfandname}
  A.~Wennman} -- {\og Universality for outliers in weakly confined coulomb-type
  systems\fg}, preprint
  \href{https://arxiv.org/abs/2104.03959v1}{arXiv:2104.03959v1}, 2021.

\bibitem[BH19]{MR3902480}
{\scshape C.~Beltr\'{a}n {\normalfont \smfandname} A.~Hardy} -- {\og Energy of
  the {C}oulomb gas on the sphere at low temperature\fg}, \emph{Arch. Ration.
  Mech. Anal.} \textbf{231} (2019), no.~3, p.~2007--2017.

\bibitem[BHS19]{zbMATH07103761}
{\scshape S.~V. {Borodachov}, D.~P. {Hardin} {\normalfont \smfandname} E.~B.
  {Saff}} -- \emph{{Discrete energy on rectifiable sets}}, New York, NY:
  Springer, 2019 (English).

\bibitem[BL15]{MR3429164}
{\scshape X.~Blanc {\normalfont \smfandname} M.~Lewin} -- {\og The
  crystallization conjecture: a review\fg}, \emph{EMS Surv. Math. Sci.}
  \textbf{2} (2015), no.~2, p.~225--306.

\bibitem[BLS18]{MR3885548}
{\scshape F.~Bekerman, T.~Lebl\'{e} {\normalfont \smfandname} S.~Serfaty} --
  {\og C{LT} for fluctuations of {$\beta$}-ensembles with general
  potential\fg}, \emph{Electron. J. Probab.} \textbf{23} (2018), p.~Paper no.
  115, 31.

\bibitem[Bor11]{ECP2011-10}
{\scshape C.~Bordenave} -- {\og On the spectrum of sum and product of
  non-{H}ermitian random matrices\fg}, \emph{Electronic Communications in
  Probability} \textbf{16} (2011), p.~104--113.

\bibitem[Bou15]{MR3443870}
{\scshape S.~Boucksom} -- {\og Limite thermodynamique et th\'{e}orie du
  potentiel\fg}, \emph{Gaz. Math.} (2015), no.~146, p.~16--26.

\bibitem[Bre67]{MR0259146}
{\scshape M.~Brelot} -- \emph{Lectures on potential theory}, Notes by K. N.
  Gowrisankaran and M. K. Venkatesha Murthy. Second edition, revised and
  enlarged with the help of S. Ramaswamy. Tata Institute of Fundamental
  Research Lectures on Mathematics, No. 19, Tata Institute of Fundamental
  Research, Bombay, 1967.

\bibitem[But17]{butheze}
{\scshape R.~Butez} -- {\og Polynômes aléatoires, gaz de {C}oulomb, et
  matrices aléatoires ({R}andom polynomials, {C}oulomb gas and random
  matrices)\fg}, \smfphdthesisname, Paris-Dauphine / PSL, 2017.

\bibitem[BW11]{MR2932622}
{\scshape O.~Bohigas {\normalfont \smfandname} H.~A. Weidenm\"{u}ller} -- {\og
  History---an overview\fg}, in \emph{The {O}xford handbook of random matrix
  theory}, Oxford Univ. Press, Oxford, 2011, p.~15--39.

\bibitem[CF19]{MR3911782}
{\scshape D.~Chafa\"{i} {\normalfont \smfandname} G.~Ferr\'{e}} -- {\og
  Simulating {C}oulomb and {L}og-{G}ases with {H}ybrid {M}onte {C}arlo
  {A}lgorithms\fg}, \emph{J. Stat. Phys.} \textbf{174} (2019), no.~3,
  p.~692--714.

\bibitem[CFS21]{zbMATH07293730}
{\scshape D.~{Chafa\"{\i}}, G.~{Ferr\'e} {\normalfont \smfandname} G.~{Stoltz}}
  -- {\og {Coulomb gases under constraint: some theoretical and numerical
  results}\fg}, \emph{{SIAM J. Math. Anal.}} \textbf{53} (2021), no.~1,
  p.~181--220 (English).

\bibitem[CGZ14]{MR3262506}
{\scshape D.~Chafa\"{i}, N.~Gozlan {\normalfont \smfandname} P.-A. Zitt} --
  {\og First-order global asymptotics for confined particles with singular pair
  repulsion\fg}, \emph{Ann. Appl. Probab.} \textbf{24} (2014), no.~6,
  p.~2371--2413.

\bibitem[CGZJ20a]{chafai-garcia-zelada-jung-1d}
{\scshape D.~Chafaï, D.~García-Zelada {\normalfont \smfandname} P.~Jung} --
  {\og At the edge of a one-dimensional jellium\fg}, preprint
  \href{https://arxiv.org/abs/2012.04633}{2012.04633} to appear in {B}ernoulli,
  2020.

\bibitem[CGZJ20b]{doi:10.1063/1.5126724}
{\scshape D.~Chafaï, D.~García-Zelada {\normalfont \smfandname} P.~Jung} --
  {\og Macroscopic and edge behavior of a planar jellium\fg}, \emph{Journal of
  Mathematical Physics} \textbf{61} (2020), no.~3, p.~033304.

\bibitem[Cha14]{blog}
{\scshape D.~Chafaï} -- {\og {W}igner about level spacing and {W}ishart\fg},
  blogpost
  \url{http://djalil.chafai.net/blog/2014/09/26/wigner-about-level-spacing-and-wishart/},
  2014.

\bibitem[Cha15]{MR3434251}
{\scshape D.~Chafa\"{\i}} -- {\og From {B}oltzmann to random matrices and
  beyond\fg}, \emph{Ann. Fac. Sci. Toulouse Math. (6)} \textbf{24} (2015),
  no.~4, p.~641--689.

\bibitem[Cha19a]{blog-vw}
{\scshape D.~Chafaï} -- {\og An unexpected distribution\fg}, blogpost
  \url{https://djalil.chafai.net/blog/2019/12/15/an-unexpected-distribution/},
  2019.

\bibitem[Cha19b]{MR4026615}
{\scshape S.~Chatterjee} -- {\og Rigidity of the three-dimensional hierarchical
  {C}oulomb gas\fg}, \emph{Probab. Theory Related Fields} \textbf{175} (2019),
  no.~3-4, p.~1123--1176.

\bibitem[CHM18]{MR3820329}
{\scshape D.~Chafa\"{i}, A.~Hardy {\normalfont \smfandname} M.~Ma\"{i}da} --
  {\og Concentration for {C}oulomb gases and {C}oulomb transport
  inequalities\fg}, \emph{J. Funct. Anal.} \textbf{275} (2018), no.~6,
  p.~1447--1483.

\bibitem[Cho54]{MR0080760}
{\scshape G.~Choquet} -- {\og Theory of capacities\fg}, \emph{Ann. Inst.
  Fourier, Grenoble} \textbf{5} (1953--1954), p.~131--295 (1955).

\bibitem[Cho90]{Choquet1990}
\bysame , {\og La vie et l'oeuvre de {M}arcel {B}relot (1903-1987)\fg},
  \emph{Cahiers du séminaire d'histoire des mathématiques} \textbf{11}
  (1990), p.~1--31 (fre).

\bibitem[CL95]{MR3155254}
{\scshape O.~Costin {\normalfont \smfandname} J.~L. Lebowitz} -- {\og Gaussian
  fluctuation in random matrices\fg}, \emph{Phys. Rev. Lett.} \textbf{75}
  (1995), no.~1, p.~69--72.

\bibitem[CL20]{MR4175749}
{\scshape D.~Chafa\"{\i} {\normalfont \smfandname} J.~Lehec} -- {\og On
  {P}oincar\'{e} and logarithmic {S}obolev inequalities for a class of singular
  {G}ibbs measures\fg}, in \emph{Geometric aspects of functional analysis.
  {V}ol. {I}}, Lecture Notes in Math., vol. 2256, Springer, Cham, [2020]
  \copyright 2020, p.~219--246.

\bibitem[CLMP92]{MR1145596}
{\scshape E.~Caglioti, P.-L. Lions, C.~Marchioro {\normalfont \smfandname}
  M.~Pulvirenti} -- {\og A special class of stationary flows for
  two-dimensional {E}uler equations: a statistical mechanics description\fg},
  \emph{Comm. Math. Phys.} \textbf{143} (1992), no.~3, p.~501--525.

\bibitem[CMMOC18]{MR3831027}
{\scshape T.~Carroll, J.~Marzo, X.~Massaneda {\normalfont \smfandname}
  J.~Ortega-Cerd\`a} -- {\og Equidistribution and {$\beta$}-ensembles\fg},
  \emph{Ann. Fac. Sci. Toulouse Math. (6)} \textbf{27} (2018), no.~2,
  p.~377--387.

\bibitem[CP14]{MR3215627}
{\scshape D.~Chafa\"{i} {\normalfont \smfandname} S.~P\'{e}ch\'{e}} -- {\og A
  note on the second order universality at the edge of {C}oulomb gases on the
  plane\fg}, \emph{J. Stat. Phys.} \textbf{156} (2014), no.~2, p.~368--383.

\bibitem[CZ98]{MR1643533}
{\scshape L.-L. Chau {\normalfont \smfandname} O.~Zaboronsky} -- {\og On the
  structure of correlation functions in the normal matrix model\fg},
  \emph{Comm. Math. Phys.} \textbf{196} (1998), no.~1, p.~203--247.

\bibitem[{de }37]{zbMATH03031018}
{\scshape C.-J. {de la Vall\'ee Poussin}} -- {\og {Les nouvelles m\'ethodes de
  la th\'eorie du potentiel et le probl\`eme g\'en\'eralise de Dirichlet.}\fg},
  {Paris: Hermann \& Cie. 47 p. (1937).}, 1937.

\bibitem[DE02]{MR1936554}
{\scshape I.~Dumitriu {\normalfont \smfandname} A.~Edelman} -- {\og Matrix
  models for beta ensembles\fg}, \emph{J. Math. Phys.} \textbf{43} (2002),
  no.~11, p.~5830--5847.

\bibitem[Dei99]{MR1677884}
{\scshape P.~A. Deift} -- \emph{Orthogonal polynomials and random matrices: a
  {R}iemann-{H}ilbert approach}, Courant Lecture Notes in Mathematics, vol.~3,
  New York University, Courant Institute of Mathematical Sciences, New York;
  American Mathematical Society, Providence, RI, 1999.

\bibitem[DG09]{MR2514781}
{\scshape P.~Deift {\normalfont \smfandname} D.~Gioev} -- \emph{Random matrix
  theory: invariant ensembles and universality}, Courant Lecture Notes in
  Mathematics, vol.~18, Courant Institute of Mathematical Sciences, New York;
  American Mathematical Society, Providence, RI, 2009.

\bibitem[DHLM21]{MR4178183}
{\scshape D.~Dereudre, A.~Hardy, T.~Lebl\'{e} {\normalfont \smfandname}
  M.~Ma\"{\i}da} -- {\og D{LR} equations and rigidity for the sine-beta
  process\fg}, \emph{Comm. Pure Appl. Math.} \textbf{74} (2021), no.~1,
  p.~172--222.

\bibitem[DLR20]{zbMATH07206384}
{\scshape P.~{Dupuis}, V.~{Laschos} {\normalfont \smfandname} K.~{Ramanan}} --
  {\og {Large deviations for configurations generated by Gibbs distributions
  with energy functionals consisting of singular interaction and weakly
  confining potentials}\fg}, \emph{{Electron. J. Probab.}} \textbf{25} (2020),
  p.~41 (English), Id/No 46.

\bibitem[DM78]{dellacheriemeyer}
{\scshape C.~Dellacherie {\normalfont \smfandname} P.-A. Meyer} --
  \emph{Probabilities and potential}, North-Holland Mathematics Studies,
  vol.~29, North-Holland Publishing Co., Amsterdam, 1978.

\bibitem[Doo01]{MR1814344}
{\scshape J.~L. Doob} -- \emph{Classical potential theory and its probabilistic
  counterpart}, Classics in Mathematics, Springer-Verlag, Berlin, 2001, Reprint
  of the 1984 edition.

\bibitem[DRSV17]{zbMATH06741334}
{\scshape B.~{Duplantier}, R.~{Rhodes}, S.~{Sheffield} {\normalfont
  \smfandname} V.~{Vargas}} -- {\og {Log-correlated Gaussian fields: an
  overview}\fg}, in \emph{{Geometry, analysis and probability. In Honor of
  Jean-Michel Bismut. Selected papers based on the presentations at the
  conference `Control, index, traces and determinants -- the journey of a
  probabilist', Orsay, France, May 27--31, 2013}}, Basel:
  Birkh\"auser/Springer, 2017, p.~191--216 (English).

\bibitem[Dub18a]{MR3878136}
{\scshape G.~Dubach} -- {\og Powers of {G}inibre eigenvalues\fg},
  \emph{Electron. J. Probab.} \textbf{23} (2018), p.~Paper No. 111, 31.

\bibitem[Dub18b]{dubach-quaternions}
\bysame , {\og Symmetries of the {Q}uaternionic {G}inibre {E}nsemble\fg},
  preprint \href{https://arxiv.org/abs/1811.03724}{arXiv:1811.03724}, 2018.

\bibitem[Dys62a]{MR0148397}
{\scshape F.~J. Dyson} -- {\og A {B}rownian-motion model for the eigenvalues of
  a random matrix\fg}, \emph{J. Mathematical Phys.} \textbf{3} (1962),
  p.~1191--1198.

\bibitem[Dys62b]{MR0143556}
\bysame , {\og Statistical theory of the energy levels of complex systems.
  {I}\fg}, \emph{J. Mathematical Phys.} \textbf{3} (1962), p.~140--156.

\bibitem[Dys62c]{MR0177643}
\bysame , {\og The threefold way. {A}lgebraic structure of symmetry groups and
  ensembles in quantum mechanics\fg}, \emph{J. Mathematical Phys.} \textbf{3}
  (1962), p.~1199--1215.

\bibitem[Ede97]{MR1437734}
{\scshape A.~Edelman} -- {\og The probability that a random real {G}aussian
  matrix has {$k$} real eigenvalues, related distributions, and the circular
  law\fg}, \emph{J. Multivariate Anal.} \textbf{60} (1997), no.~2, p.~203--232.

\bibitem[EF05]{MR2172690}
{\scshape P.~Elbau {\normalfont \smfandname} G.~Felder} -- {\og Density of
  eigenvalues of random normal matrices\fg}, \emph{Comm. Math. Phys.}
  \textbf{259} (2005), no.~2, p.~433--450.

\bibitem[ER05]{MR2168344}
{\scshape A.~Edelman {\normalfont \smfandname} N.~R. Rao} -- {\og Random matrix
  theory\fg}, \emph{Acta Numer.} \textbf{14} (2005), p.~233--297.

\bibitem[EY17]{MR3699468}
{\scshape L.~Erd\H{o}s {\normalfont \smfandname} H.-T. Yau} -- \emph{A
  dynamical approach to random matrix theory}, Courant Lecture Notes in
  Mathematics, vol.~28, Courant Institute of Mathematical Sciences, New York;
  American Mathematical Society, Providence, RI, 2017.

\bibitem[FG16]{MR3646880}
{\scshape A.~Figalli {\normalfont \smfandname} A.~Guionnet} -- {\og
  Universality in several-matrix models via approximate transport maps\fg},
  \emph{Acta Math.} \textbf{217} (2016), no.~1, p.~81--176.

\bibitem[FK09]{forrester-krishnapur}
{\scshape P.~J. Forrester {\normalfont \smfandname} M.~Krishnapur} -- {\og
  Derivation of an eigenvalue probability density function relating to the
  poincaré disk\fg}, \emph{J. of Physics A: Math. and Theor.} \textbf{42}
  (2009), no.~38, p.~385204.

\bibitem[For10]{MR2641363}
{\scshape P.~J. Forrester} -- \emph{Log-gases and random matrices}, London
  Mathematical Society Monographs Series, vol.~34, Princeton University Press,
  Princeton, NJ, 2010.

\bibitem[Fro35]{frostman}
{\scshape O.~Frostman} -- {\og {P}otentiel d'\'{E}quilibre et {C}apacit\'e des
  {E}nsembles\fg}, \smfphdthesisname, Facult\'e des sciences de {L}und, 1935.

\bibitem[FW08]{zbMATH05348712}
{\scshape P.~J. {Forrester} {\normalfont \smfandname} S.~O. {Warnaar}} -- {\og
  {The importance of the Selberg integral}\fg}, \emph{{Bull. Am. Math. Soc.,
  New Ser.}} \textbf{45} (2008), no.~4, p.~489--534 (English).

\bibitem[{Gar}19]{zbMATH07133725}
{\scshape D.~{Garc\'{\i}a-Zelada}} -- {\og {A large deviation principle for
  empirical measures on Polish spaces: application to singular Gibbs measures
  on manifolds}\fg}, \emph{{Ann. Inst. Henri Poincar\'e, Probab. Stat.}}
  \textbf{55} (2019), no.~3, p.~1377--1401 (English).

\bibitem[Gin65]{MR0173726}
{\scshape J.~Ginibre} -- {\og Statistical ensembles of complex, quaternion, and
  real matrices\fg}, \emph{J. Mathematical Phys.} \textbf{6} (1965),
  p.~440--449.

\bibitem[Gir05]{MR2180172}
{\scshape S.~M. Girvin} -- {\og Introduction to the fractional quantum {H}all
  effect\fg}, in \emph{The quantum {H}all effect}, Prog. Math. Phys., vol.~45,
  Birkh\"{a}user, Basel, 2005, p.~133--162.

\bibitem[GS20]{garban-sepulveda}
{\scshape C.~Garban {\normalfont \smfandname} A.~Sepúlveda} -- {\og
  Statistical reconstruction of the {G}aussian free field and {K}{T}
  transition\fg}, preprint
  \href{https://arxiv.org/abs/2002.12284}{arXiv:2002.12284}, 2020.

\bibitem[GT11]{gotze-tikhomirov}
{\scshape F.~Götze {\normalfont \smfandname} A.~Tikhomirov} -- {\og On the
  asymptotic spectrum of products of independent random matrices\fg}, preprint
  \href{https://arxiv.org/abs/1012.2710v3}{arXiv:1012.2710v3}, 2011.

\bibitem[GZ18]{david-edge}
{\scshape D.~García-Zelada} -- {\og Edge fluctuations for random normal matrix
  ensembles\fg}, preprint
  \href{https://arxiv.org/abs/1812.11170v3}{arXiv:1812.11170v3}, 2018.

\bibitem[GZ19a]{MR3933036}
{\scshape D.~Garc\'{\i}a-Zelada} -- {\og Concentration for {C}oulomb gases on
  compact manifolds\fg}, \emph{Electron. Commun. Probab.} \textbf{24} (2019),
  p.~Paper No. 12, 18.

\bibitem[GZ19b]{phdavid}
{\scshape D.~García-Zelada} -- {\og Aspects géométriques et probabilistes
  des gaz de {C}oulomb ({G}eometric and probabilistic aspects of {C}oulomb
  gases)\fg}, \smfphdthesisname, Paris-Dauphine / PSL, 2019.

\bibitem[Har12]{MR2926763}
{\scshape A.~Hardy} -- {\og A note on large deviations for 2{D} {C}oulomb gas
  with weakly confining potential\fg}, \emph{Electron. Commun. Probab.}
  \textbf{17} (2012), p.~no. 19, 12.

\bibitem[Hel14]{MR3308615}
{\scshape L.~L. Helms} -- \emph{Potential theory}, second \smfedname,
  Universitext, Springer, London, 2014.

\bibitem[Hir78]{MR521767}
{\scshape F.~Hirsch} -- {\og Op\'{e}rateurs carr\'{e} du champ (d'apr\`es {J}.
  {P}. {R}oth)\fg}, in \emph{S\'{e}minaire {B}ourbaki, 29e ann\'{e}e
  (1976/77)}, Lecture Notes in Math., vol. 677, Springer, Berlin, 1978, p.~Exp.
  No. 501, pp. 167--182.

\bibitem[HKPV09]{MR2552864}
{\scshape J.~B. Hough, M.~Krishnapur, Y.~Peres {\normalfont \smfandname}
  B.~Vir{\'a}g} -- \emph{Zeros of {G}aussian analytic functions and
  determinantal point processes}, University Lecture Series, vol.~51, AMS,
  Providence, RI, 2009.

\bibitem[HL21]{MR4211032}
{\scshape A.~Hardy {\normalfont \smfandname} G.~Lambert} -- {\og C{LT} for
  circular beta-ensembles at high temperature\fg}, \emph{J. Funct. Anal.}
  \textbf{280} (2021), no.~7, p.~Paper No. 108869, 40.

\bibitem[HP00]{MR1746976}
{\scshape F.~Hiai {\normalfont \smfandname} D.~Petz} -- \emph{The semicircle
  law, free random variables and entropy}, Mathematical Surveys and Monographs,
  vol.~77, American Mathematical Society, Providence, RI, 2000.

\bibitem[Hug06]{hughes2006theoretical}
{\scshape R.~Hughes} -- {\og {Theoretical practice: the Bohm-Pines
  quartet}\fg}, \emph{Perspectives on science} \textbf{14} (2006), no.~4,
  p.~457--524.

\bibitem[Hwa86]{MR841088}
{\scshape C.-R. Hwang} -- {\og A brief survey on the spectral radius and the
  spectral distribution of large random matrices with i.i.d.\ entries\fg}, in
  \emph{Random matrices and their applications (Brunswick, Maine, 1984)},
  Contemp. Math., vol.~50, Amer. Math. Soc., Providence, RI, 1986, p.~145--152.

\bibitem[Joh98]{MR1487983}
{\scshape K.~Johansson} -- {\og On fluctuations of eigenvalues of random
  {H}ermitian matrices\fg}, \emph{Duke Math. J.} \textbf{91} (1998), no.~1,
  p.~151--204.

\bibitem[Joh07]{MR2288065}
\bysame , {\og From {G}umbel to {T}racy-{W}idom\fg}, \emph{Probab. Theory
  Related Fields} \textbf{138} (2007), no.~1-2, p.~75--112.

\bibitem[JQ17]{MR3615091}
{\scshape T.~Jiang {\normalfont \smfandname} Y.~Qi} -- {\og Spectral radii of
  large non-{H}ermitian random matrices\fg}, \emph{J. Theoret. Probab.}
  \textbf{30} (2017), no.~1, p.~326--364.

\bibitem[{Kel}67]{zbMATH03244821}
{\scshape O.~D. {Kellogg}} -- {\og {Foundations of potential theory.}\fg},
  {Berlin-Heidelberg-New York: Springer-Verlag 1967. IX, 384 p. with 30 fig.
  (1967).}, 1967.

\bibitem[Kos92]{MR1148410}
{\scshape E.~Kostlan} -- {\og On the spectra of {G}aussian matrices\fg},
  \emph{Linear Algebra Appl.} \textbf{162/164} (1992), p.~385--388, Directions
  in matrix theory (Auburn, AL, 1990).

\bibitem[KP17]{kharash-peled}
{\scshape V.~Kharash {\normalfont \smfandname} R.~Peled} -- {\og The
  {F}röhlich-{S}pencer proof of the {B}erezinskii-{K}osterlitz-{T}houless
  transition\fg}, preprint
  \href{https://arxiv.org/abs/1711.04720}{arXiv:1711.04720}, 2017.

\bibitem[Kri09]{MR2489167}
{\scshape M.~Krishnapur} -- {\og From random matrices to random analytic
  functions\fg}, \emph{Ann. Probab.} \textbf{37} (2009), no.~1, p.~314--346.

\bibitem[KS11]{MR2932638}
{\scshape B.~A. Khoruzhenko {\normalfont \smfandname} H.-J. Sommers} -- {\og
  Non-{H}ermitian ensembles\fg}, in \emph{The {O}xford handbook of random
  matrix theory}, Oxford Univ. Press, Oxford, 2011, p.~376--397.

\bibitem[KZ13]{MR3127891}
{\scshape Z.~Kabluchko {\normalfont \smfandname} D.~Zaporozhets} -- {\og Roots
  of random polynomials whose coefficients have logarithmic tails\fg},
  \emph{Ann. Probab.} \textbf{41} (2013), no.~5, p.~3542--3581.

\bibitem[KZ14]{MR3262481}
\bysame , {\og Asymptotic distribution of complex zeros of random analytic
  functions\fg}, \emph{Ann. Probab.} \textbf{42} (2014), no.~4, p.~1374--1395.

\bibitem[LACTMS19]{PhysRevA.99.021602}
{\scshape B.~Lacroix-A-Chez-Toine, S.~N. Majumdar {\normalfont \smfandname}
  G.~Schehr} -- {\og Rotating trapped fermions in two dimensions and the
  complex ginibre ensemble: Exact results for the entanglement entropy and
  number variance\fg}, \emph{Phys. Rev. A} \textbf{99} (2019), p.~021602.

\bibitem[Lan72]{MR0350027}
{\scshape N.~S. Landkof} -- \emph{Foundations of modern potential theory},
  Springer-Verlag, New York-Heidelberg, 1972, Translated from the Russian by A.
  P. Doohovskoy, Die Grundlehren der mathematischen Wissenschaften, Band 180.

\bibitem[Lau87]{Laughlin1987}
{\scshape R.~B. Laughlin} -- {\og Elementary theory: the incompressible quantum
  fluid\fg}, in \emph{The Quantum Hall Effect} (R.~E. Prange {\normalfont
  \smfandname} S.~M. Girvin, \smfedsname), Springer US, New York, NY, 1987,
  p.~233--301.

\bibitem[Lew21]{lewin-mega}
{\scshape M.~Lewin} -- {\og {R}iesz and {C}oulomb gases: what's known and
  unknown\fg}, Survey talk at ``Séminaire {M}atrices {E}t {G}raphes
  {A}léatoires'' ({MEGA}), {I}nstitut {H}enri {P}oincaré ({IHP}),
  \url{https://www.ceremade.dauphine.fr/dokuwiki/mega:seminaire}, 2021.

\bibitem[LL01]{MR1817225}
{\scshape E.~H. Lieb {\normalfont \smfandname} M.~Loss} -- \emph{Analysis},
  second \smfedname, Graduate Studies in Mathematics, vol.~14, American
  Mathematical Society, Providence, RI, 2001.

\bibitem[LS17]{MR3735628}
{\scshape T.~Lebl\'{e} {\normalfont \smfandname} S.~Serfaty} -- {\og Large
  deviation principle for empirical fields of log and {R}iesz gases\fg},
  \emph{Invent. Math.} \textbf{210} (2017), no.~3, p.~645--757.

\bibitem[LS18]{MR3788208}
\bysame , {\og Fluctuations of two dimensional {C}oulomb gases\fg}, \emph{Geom.
  Funct. Anal.} \textbf{28} (2018), no.~2, p.~443--508.

\bibitem[LZ20]{leble-zeitouni}
{\scshape T.~Leblé {\normalfont \smfandname} O.~Zeitouni} -- {\og A local
  {CLT} for linear statistics of 2{D} {C}oulomb gases\fg},
  \href{https://arxiv.org/abs/2005.12163}{arXiv:2005.12163}, 2020.

\bibitem[Meh04]{MR2129906}
{\scshape M.~L. Mehta} -- \emph{Random matrices}, third \smfedname, Pure and
  Applied Mathematics, vol. 142, Acad. Press, 2004.

\bibitem[MMS14]{MR3201924}
{\scshape M.~Ma\"{i}da {\normalfont \smfandname} E.~Maurel-Segala} -- {\og Free
  transport-entropy inequalities for non-convex potentials and application to
  concentration for random matrices\fg}, \emph{Probab. Theory Related Fields}
  \textbf{159} (2014), no.~1-2, p.~329--356.

\bibitem[MS19]{zbMATH07107321}
{\scshape E.~{Meckes} {\normalfont \smfandname} K.~{Stewart}} -- {\og {On the
  eigenvalues of truncations of random unitary matrices}\fg}, \emph{{Electron.
  Commun. Probab.}} \textbf{24} (2019), p.~12 (English), Id/No 57.

\bibitem[PG20]{padilla-garza-concentration}
{\scshape D.~Padilla-Garza} -- {\og Concentration inequality around the thermal
  equilibrium measure of {C}oulomb gases\fg}, preprint
  \href{httsp://arxiv.org/abs/2010.00194}{arXiv:2010.00194}, 2020.

\bibitem[PH98]{MR1606719}
{\scshape D.~Petz {\normalfont \smfandname} F.~Hiai} -- {\og Logarithmic energy
  as an entropy functional\fg}, in \emph{Advances in differential equations and
  mathematical physics ({A}tlanta, {GA}, 1997)}, Contemp. Math., vol. 217,
  Amer. Math. Soc., Providence, RI, 1998, p.~205--221.

\bibitem[PR05]{MR2198697}
{\scshape D.~Petz {\normalfont \smfandname} J.~R\'{e}ffy} -- {\og Large
  deviation for the empirical eigenvalue density of truncated {H}aar unitary
  matrices\fg}, \emph{Probab. Theory Related Fields} \textbf{133} (2005),
  no.~2, p.~175--189.

\bibitem[PS11]{MR2808038}
{\scshape L.~Pastur {\normalfont \smfandname} M.~Shcherbina} --
  \emph{Eigenvalue distribution of large random matrices}, Mathematical Surveys
  and Monographs, vol. 171, American Mathematical Society, Providence, RI,
  2011.

\bibitem[PS20]{zbMATH07204762}
{\scshape M.~{Petrache} {\normalfont \smfandname} S.~{Serfaty}} -- {\og
  {Crystallization for Coulomb and Riesz interactions as a consequence of the
  Cohn-Kumar conjecture}\fg}, \emph{{Proc. Am. Math. Soc.}} \textbf{148}
  (2020), no.~7, p.~3047--3057 (English).

\bibitem[Rid03]{MR1986426}
{\scshape B.~Rider} -- {\og A limit theorem at the edge of a non-{H}ermitian
  random matrix ensemble\fg}, \emph{J. Phys. A} \textbf{36} (2003), no.~12,
  p.~3401--3409, Random matrix theory.

\bibitem[Rot76]{MR0448158}
{\scshape J.-P. Roth} -- {\og Op\'{e}rateurs dissipatifs et semi-groupes dans
  les espaces de fonctions continues\fg}, \emph{Ann. Inst. Fourier (Grenoble)}
  \textbf{26} (1976), no.~4, p.~ix, 1--97.

\bibitem[Rou15]{rougerie}
{\scshape N.~Rougerie} -- {\og De {F}inetti theorems, mean-field limits and
  {B}ose--{E}instein condensation\fg}, Cours Peccot lecture notes
  \href{http://arxiv.org/abs/1506.05263v1}{arXiv:1506.05263v1}, 2015.

\bibitem[Roy07]{MR2352327}
{\scshape G.~Royer} -- \emph{An initiation to logarithmic {S}obolev
  inequalities}, SMF/AMS Texts and Monographs, vol.~14, American Mathematical
  Society, Providence, RI; Soci\'{e}t\'{e} Math\'{e}matique de France, Paris,
  2007, Translated from the 1999 French original by Donald Babbitt.

\bibitem[RS93]{MR1217451}
{\scshape L.~C.~G. Rogers {\normalfont \smfandname} Z.~Shi} -- {\og Interacting
  {B}rownian particles and the {W}igner law\fg}, \emph{Probab. Theory Related
  Fields} \textbf{95} (1993), no.~4, p.~555--570.

\bibitem[RS16]{MR3455593}
{\scshape N.~Rougerie {\normalfont \smfandname} S.~Serfaty} -- {\og
  Higher-dimensional {C}oulomb gases and renormalized energy functionals\fg},
  \emph{Comm. Pure Appl. Math.} \textbf{69} (2016), no.~3, p.~519--605.

\bibitem[RV07]{MR2361453}
{\scshape B.~Rider {\normalfont \smfandname} B.~Vir{\'a}g} -- {\og The noise in
  the circular law and the {G}aussian free field\fg}, \emph{Int. Math. Res.
  Not. IMRN} (2007), no.~2, p.~Art. ID rnm006, 33.

\bibitem[Seo20]{MR4179777}
{\scshape S.-M. Seo} -- {\og Edge {S}caling {L}imit of the {S}pectral {R}adius
  for {R}andom {N}ormal {M}atrix {E}nsembles at {H}ard {E}dge\fg}, \emph{J.
  Stat. Phys.} \textbf{181} (2020), no.~5, p.~1473--1489.

\bibitem[Ser15]{MR3309890}
{\scshape S.~Serfaty} -- \emph{Coulomb gases and {G}inzburg-{L}andau vortices},
  Zurich Lectures in Advanced Mathematics, European Mathematical Society (EMS),
  Z\"{u}rich, 2015.

\bibitem[Ser18a]{MR3837028}
{\scshape S.~Serfaty} -- {\og Syst\`emes de points en interaction
  coulombienne\fg}, \emph{Gaz. Math.} (2018), no.~157, p.~29--37.

\bibitem[Ser18b]{MR3888701}
{\scshape S.~Serfaty} -- {\og Systems of points with {C}oulomb
  interactions\fg}, \emph{Eur. Math. Soc. Newsl.} (2018), no.~110, p.~16--21.

\bibitem[Ser20a]{serfaty-clt}
\bysame , {\og Gaussian fluctuations and free energy expansion for 2{D} and
  3{D} {C}oulomb gases at any temperature\fg}, preprint
  \href{https://arxiv.org/abs/2003.11704}{arXiv:2003.11704}, 2020.

\bibitem[Ser20b]{MR4158670}
\bysame , {\og Mean field limit for {C}oulomb-type flows\fg}, \emph{Duke Math.
  J.} \textbf{169} (2020), no.~15, p.~2887--2935, With an appendix by Mitia
  Duerinckx and Serfaty.

\bibitem[ST97]{MR1485778}
{\scshape E.~B. Saff {\normalfont \smfandname} V.~Totik} -- \emph{Logarithmic
  potentials with external fields}, Fundamental Principles of Mathematical
  Sciences, vol. 316, Springer-Verlag, 1997, Appendix B by Thomas Bloom.

\bibitem[Tho04]{doi:10.1080/14786440409463107}
{\scshape J.~J. Thomson} -- {\og {X}{X}{I}{V}. {O}n the structure of the atom:
  an investigation of the stability and periods of oscillation of a number of
  corpuscles arranged at equal intervals around the circumference of a circle;
  with application of the results to the theory of atomic structure\fg},
  \emph{The London, Edinburgh, and Dublin Philosophical Magazine and Journal of
  Science} \textbf{7} (1904), no.~39, p.~237--265.

\bibitem[Voi93]{MR1228526}
{\scshape D.~Voiculescu} -- {\og The analogues of entropy and of {F}isher's
  information measure in free probability theory. {I}\fg}, \emph{Comm. Math.
  Phys.} \textbf{155} (1993), no.~1, p.~71--92.

\bibitem[Voi94]{MR1296352}
\bysame , {\og The analogues of entropy and of {F}isher's information measure
  in free probability theory. {II}\fg}, \emph{Invent. Math.} \textbf{118}
  (1994), no.~3, p.~411--440.

\bibitem[Wig38]{TF9383400678}
{\scshape E.~Wigner} -- {\og Effects of the electron interaction on the energy
  levels of electrons in metals\fg}, \emph{Trans. Faraday Soc.} \textbf{34}
  (1938), p.~678--685.

\bibitem[Yuk98]{MR1632875}
{\scshape J.~E. Yukich} -- \emph{Probability theory of classical {E}uclidean
  optimization problems}, Lecture Notes in Mathematics, vol. 1675,
  Springer-Verlag, Berlin, 1998.

\bibitem[ZS00]{MR1748745}
{\scshape K.~\.{Z}yczkowski {\normalfont \smfandname} H.-J. Sommers} -- {\og
  Truncations of random unitary matrices\fg}, \emph{J. Phys. A} \textbf{33}
  (2000), no.~10, p.~2045--2057.

\end{thebibliography}

\end{document}